\documentclass[a4paper,10pt]{article}
\usepackage{amsmath}
\usepackage{amsfonts}
\usepackage{amssymb}
\usepackage{amsthm}
\usepackage{graphicx}
\usepackage{subfigure}
\usepackage{geometry}
\usepackage[utf8]{inputenc}
\usepackage{float}
\usepackage{pstricks-add}
\usepackage{tikz}
\usetikzlibrary{patterns}

\newtheorem{lem}{Lemma}
\newtheorem{thrm}{Theorem}
\newtheorem{cor}{Corollary}

\newtheorem{defi}{Definition}
\newtheorem{rem}{Remark}

\newcommand{\ds}{\displaystyle}

\def\R{\mathbb R}

\newcommand{\M}{\mathfrak{M}}
\newcommand{\Mie}{\mathfrak{M}^*}
\newcommand{\dM}{\dr\M}
\newcommand{\dMie}{\dr\Mie}

\newcommand{\T}{\mathcal{T}}

\newcommand{\E}{\mathcal{E}}

\newcommand{\dr}{\partial}
\newcommand{\petitt}{{\scriptscriptstyle\mathcal{T}}}
\newcommand{\DD}{\mathfrak{D}}

\newcommand{\D}{{\mathcal{D}}}

\def\k{{ \mathcal{K}}}
\def\l{{ \mathcal{L}}}
\newcommand{\ke}{{ \mathcal{K}^*}}
\def\le{{\mathcal{L}^*}}
\newcommand{\sig}{ \sigma}
\newcommand{\sige}{\sigma^*}
\newcommand{\Dsig}{{\mathcal{D}_{\sigma,\sigma^*}}}

\newcommand{\petitD}{{\scriptscriptstyle\mathcal{D}}}

\newcommand{\uk}{u_{\k}}
\newcommand{\ul}{u_{\l}}
\newcommand{\uke}{u_{\ke}}
\newcommand{\ule}{u_{\le}}

\newcommand{\xk}{x_{\k}}
\newcommand{\xl}{x_{\l}}

\newcommand{\xke}{x_{\ke}}
\newcommand{\xle}{x_{\le}}

\newcommand{\mk}{{\rm m}({\k})}
\newcommand{\mke}{{\rm m}({\ke})}
\newcommand{\md}{{\rm m}({\D})}

\newcommand{\msige}{{\rm m}({\sige})}

\newcommand{\petitk}{{\scriptscriptstyle \mathcal{K}}}
\newcommand{\petitl}{{\scriptscriptstyle \mathcal{L}}}

\newcommand{\petitke}{{\scriptscriptstyle \mathcal{K}^*}}
\newcommand{\petitle}{{\scriptscriptstyle \mathcal{L}^*}}

\newcommand{\tkele}{\boldsymbol{{\boldsymbol{\tau}}}_{\boldsymbol{\petitke,\petitle}}}
\newcommand{\tkl}{\boldsymbol{{\boldsymbol{\tau}}}_{\boldsymbol{\petitk,\petitl}}}

\newcommand{\nksig}{{\mathbf{n}}_{\sig\petitk}}

\newcommand{\nkesige}{\boldsymbol{{\mathbf{n}}}_{\sige\petitke}}

\newcommand{\dkel}{d_{\ke,\l}}
\newcommand{\dlel}{d_{\le,\l}}

\newcommand{\gradD}{\nabla^{\D}}
\newcommand{\sind}{ \sin(\alpha_{\petitD})}

\newcommand{\sint}{ \sin(\alpha_{\petitt})}

\newcommand{\lv}{\left\vert}
\newcommand{\rv}{\right\vert}
\newcommand{\lV}{\left\Vert}
\newcommand{\rV}{\right\Vert}

\def\pa{\partial}

\title{On discrete functional inequalities for some finite volume schemes}
\author{Marianne BESSEMOULIN-CHATARD, Claire CHAINAIS-HILLAIRET \\ and Francis FILBET}

\begin{document}

\maketitle

\begin{abstract}
We prove several discrete Gagliardo-Nirenberg-Sobolev and Poincaré-Sobolev inequalities for some approximations with arbitrary boundary values on finite volume meshes. The keypoint of our approach is to use the continuous embedding of the space $BV(\Omega)$ into $L^{N/(N-1)}(\Omega)$ for a Lipschitz domain $ \Omega \subset \mathbb{R}^{N}$,  with $N \geq 2$. Finally, we give several applications to discrete duality finite volume (DDFV) schemes which are used for the approximation of nonlinear and non isotropic elliptic and parabolic problems.  
\end{abstract}


\section{Introduction}

In this paper, we establish some discrete functional inequalities which are sometimes useful for the convergence analysis of finite volume schemes. In the continuous framework, the Gagliardo-Nirenberg-Sobolev and Poincaré-Sobolev inequalities are fundamental for the analysis of partial differential equations. They are a standard tool in existence and regularity theories for solutions. The $L^{2}$ framework is generally used for linear elliptic problems, more precisely it is a classical way to prove the coercivity of bilinear forms in $H^{1}_{0}$, which then allows to apply the Lax-Milgram theorem to prove existence of weak solutions. More generally, the $L^{p}$ framework is crucial for the study of nonlinear elliptic or parabolic equations, to obtain some energy estimates which are useful to prove existence of weak solutions. Poincaré-type inequalities are also one of the step in the study of convergence to equilibrium for kinetic equations.

\subsection{Gagliardo-Nirenberg-Sobolev and Poincaré-Sobolev inequalities}

In the continuous situation, the Poincaré-Sobolev inequality is written as follows. Let assume $N\geq 2$ and $\Omega$ be an open bounded domain of $\R^N$. If $1\leq p<N$, let $1\leq q\leq p^{*}=\frac{pN}{N-p}$; if $p\geq N$, let $1\leq q <+\infty$. Then there exists a constant $C>0$ depending on $p,q,N$ and $\Omega$ such that 
\begin{equation}
\Vert u \Vert_{L^{q}(\Omega)} \, \leq \, C \, \Vert u \Vert_{W^{1,p}(\Omega)} \quad \forall u \in W^{1,p}(\Omega),
\label{sobolevcontinu}
\end{equation}
We refer to \cite{Adams1975,Brezis2010} for a proof of this result. We also remind the Gagliardo-Nirenberg-Sobolev inequality \cite{Friedman1969,Nirenberg1959}.  If $1\leq p<N$, let $1\leq s\leq m\leq p^{*}=\frac{pN}{N-p}$; if $p\geq N$, let $1\leq s\leq m<+\infty$. Then there exists a constant $C>0$ depending on $p,s,m,N$ and $\Omega$ such that
\begin{equation}
\Vert u \Vert_{L^{m}(\Omega)} \, \leq \, C \, \Vert u \Vert^{\theta}_{W^{1,p}(\Omega)} \, \Vert u \Vert^{1-\theta}_{L^{s}(\Omega)} \quad \forall u \in W^{1,p}(\Omega)\cap L^s(\Omega),
\label{GNScontinu}
\end{equation}
where 
\begin{equation}
\theta =  \frac{\ds\frac{1}{s}-\frac{1}{m}}{\ds\frac{1}{s}+\frac{1}{N}-\frac{1}{p}}.
\label{theta:opt}
\end{equation}

The mathematical analysis of convergence and error estimates for numerical methods are performed using functional analysis tools, such as discrete Sobolev inequalities. Several Poincaré-Sobolev inequalities have been established for the finite volume schemes as well as for the finite element methods. Concerning the finite volume framework, the first estimates were obtained in the particular case $N=2$, $p=q=2$ (which is the standard Poincaré inequality) for Dirichlet boundary conditions by R. Herbin \cite{Herbin1995} and by Y. Coudi\`ere, J.-P. Vila and P. Villedieu \cite{Coudiere1999}. The idea of the proof is as follows: given an oriented direction $ \mathcal{D}$, any cell center of the mesh is connected to an upstream (with respect to $\mathcal{D}$) center of an edge of the boundary $\pa \Omega$ by a straight line of direction $ \mathcal{D}$. This connection crosses a certain number of cells and their interfaces, and this argument allows to link a norm of the piecewise constant function considered with a norm of a discrete version of its gradient. The proof requires some regularity assumptions on the mesh. The same starting point was used to generalize this result to the case of dimension $N=3$ by R. Eymard, T. Gallou\"et and R. Herbin \cite{Eymard1999,Eymard2000}, for meshes satisfying an orthogonality condition. Also the same method has been applied to get more general Poincaré inequalities for $1 \leq p=q \leq 2$ by J. Droniou, T. Gallou\"et and R. Herbin  \cite{Droniou2003}, and for $1\leq p=q<+\infty$ by B. Andreianov, M. Gutnic and P. Wittbold \cite{Andreianov2004}. Concerning Neumann boundary conditions, a discrete Poincaré-Wirtinger inequality ($p=q=2$) was established in \cite{Eymard2000,Gallouet2000} for $N=2$ or $3$ by using the same method. Then some discrete Poincaré-Sobolev inequalities \eqref{sobolevcontinu} were obtained in the case of Dirichlet boundary conditions under regularity assumption on the mesh, for $p=2$, $1 \leq q < \infty$ if $N=2$, $1 \leq q \leq 6$ if $N=3$ \cite{Coudiere2001,Eymard2000}, as well as for $1 \leq p \leq 2$, $1\leq q \leq p^{*}$ if $p<N$ \cite{Droniou2003}. In these papers, the starting point consists in establishing a bound of the $L^{N/(N-1)}$ norm with respect to a discrete $W^{1,1}$ seminorm in the spirit of the work of L. Nirenberg \cite{Nirenberg1959}. The case of Neumann boundary conditions, $p=2$, $1 \leq q<+\infty$ if $N=2$, $1 \leq q \leq p^{*}$ if $N\geq 3$, is treated in \cite{Chainais-Hillairet2011} by using the same method and a trace result.

More recently, another idea was used to prove this type of discrete inequalities: the continuous embedding of $BV(\Omega)$ into $L^{N/(N-1)}(\Omega)$ for a Lipschitz domain $\Omega$. This argument was already underlying in \cite{Chainais-Hillairet2011,Coudiere2001,Droniou2003,Eymard2000}, but it was exploited directly in the continuous setting in \cite{Filbet2006} to prove a discrete Poincaré-Sobolev inequality (\ref{sobolevcontinu}) in dimension $N=2$ with $q=2$ and $p=1$, in the case of Neumann boundary conditions. Then this method was used in \cite{Eymard2010} to prove general Poincaré-Sobolev inequalities (\ref{sobolevcontinu}) in any dimension $N \geq 1$ in the particular case of homogeneous Dirichlet boundary conditions. This result was then adapted to the case of Neumann or mixed boundary conditions by B. Andreianov, M. Bendahmane and R. Ruiz Baier in \cite{Andreianov2011}. We also mention \cite{Bouchut2011} where the continuous embedding of $BV(\mathbb{R}^{N})$ into $L^{N/(N-1)}(\mathbb{R}^{N})$ is used to establish an improved discrete Gagliardo-Nirenberg-Sobolev inequality in the whole space $ \mathbb{R}^{N}$, $N \geq 1$. 

For $p=2$, general discrete Poincaré-Sobolev inequalities have been obtained  by A.~Glitzky and J.~Griepentrog  in \cite{Glitzky2010} for Voronoi finite volume approximations in the case of arbitrary boundary conditions by using an adaptation of Sobolev's integral representation and the Voronoi property of the mesh. Concerning the finite element framework, a variant of a Poincaré-type inequality ($p=q=2$) for functions in broken Sobolev spaces was derived in \cite{Arnold1982} for $N=2$ and in \cite{Brenner2004,Vohralik2005} for $N=2$, 3. Then a generalised result was proposed in \cite{Lasis2003}, providing bounds on the $L^{q}$ norms in terms of a broken $H^{1}$ norm ($p=2$, $1 \leq q < \infty$ if $N=2$ and $1 \leq q \leq 2N/(N-2)$ if $N \geq 3$). The proof is based on elliptic regularity results and nonconforming finite element interpolants. Finally, a result in non-Hilbertian setting ($p \neq 2$) was obtained in \cite{DiPietro2010}, taking inspiration from the technique used by F. Filbet \cite{Filbet2006} and also R. Eymard,  T.~Gallou\"et and R.~Herbin \cite{Eymard2010}, namely the continuous embedding of $BV(\Omega)$ into $L^{N/(N-1)}(\Omega)$.\\

\subsection{Aim of the paper and outline }
In this paper our aim is to provide a simple proof to  discrete versions of Poincaré-Sobolev (\ref{sobolevcontinu}) and Gagliardo-Nirenberg-Sobolev (\ref{GNScontinu}) inequalities for functions coming from finite volume schemes with arbitrary boundary values. Several Poincaré-Sobolev inequalities are already proved as mentioned above but here we propose a unified result. It includes in particular the case of mixed boundary conditions. Concerning Gagliardo-Nirenberg-Sobolev inequalities, the result of F. Bouchut, R. Eymard and A. Prignet \cite{Bouchut2011} is to our knowledge the only available, and it deals with the case of the whole space $\mathbb{R}^{N}$. 

Our starting point to prove these discrete estimates is the continuous embedding of $BV(\Omega)$ into $L^{N/(N-1)}(\Omega)$, as in \cite{Filbet2006,Eymard2010,DiPietro2010,Bouchut2011}. The main difficulty appears when boundary conditions must be taken into account. In the papers mentioned previously \cite{Filbet2006,Eymard2010,DiPietro2010}, the boundary conditions are either homogeneous Dirichlet or Neumann on the whole boundary. In \cite{Bouchut2011}, the problem is considered in the whole space $ \mathbb{R}^{N}$. In the case where the function satisfies homogeneous Dirichlet boundary conditions only on a part $ \Gamma^{0} \varsubsetneq \pa \Omega $ of the boundary, we cannot use the same strategy as in \cite{Eymard2010}, which consists of extending the function considered to $\mathbb{R}^{N}$ by zero. Our idea is to thicken the boundary of $\Omega$ in order to take the mixed boundary conditions into account in this case.

The outline  of the paper is as follows. In Section \ref{sec.fs}, we first define the functional spaces: the space of finite volume approximations and the space $BV(\Omega)$. We will see that $BV(\Omega)$ is a natural space to study piecewise constant functions as finite volume approximations. In Section~\ref{sec.gc}, we do not take into account any boundary conditions and prove the discrete Poincar\'e-Sobolev inequalities (Theorem \ref{thm4}) and the discrete Gagliardo-Nirenberg-Sobolev inequalities (Theorem \ref{thm3}) in this case. These results are the discrete counterpart of \eqref{sobolevcontinu} and \eqref{GNScontinu}. They may be used for instance in the convergence analysis of finite volume schemes in the case with Neumann boundary conditions.
Then, in Section \ref{sec.dirichlet}, we consider the case where the discrete function is given by a finite volume scheme with homogeneous boundary conditions on a part of the boundary. In this case, the discrete space (for the finite volume approximations) is unchanged. However,  the discrete $W^{1,p}$ seminorm will take into account some jumps on the boundary.  We prove discrete Poincar\'e-Sobolev inequalities (Theorem \ref{thm2}) and discrete Gagliardo-Nirenberg-Sobolev inequalities (Theorem \ref{thm1}), similar to \eqref{sobolevcontinu} and \eqref{GNScontinu} but with the $W^{1,p}$ seminorm instead of the full $W^{1,p}$ norm. Finally, in Section \ref{sec.ddfv}, we show how to extend the results from Sections \ref{sec.gc} and \ref{sec.dirichlet} to finite volume approximations coming from discrete duality finite volume (DDFV) schemes. This family of schemes is mainly applied to anisotropic elliptic and parabolic problems. This method can be applied to a wide class of 2D meshes (but also 3D \cite{Coudi`ere2010}) and inherits the main qualitative properties of the continuous problem: monotonicity, coercivity, variational formulation, etc...

%
%
%
%
%

\section{Functional spaces}\label{sec.fs}

\subsection{The space of finite volume approximations}

We now introduce the discrete settings, including notations and assumptions on the meshes and definitions of the discrete norms. Let $ \Omega$ be an open bounded polyhedral susbset (Lipschitz domain) of $ \mathbb{R}^{N}$, $N \geq 2$, and $ \Gamma := \pa \Omega$ its boundary. In the sequel, we denote by d the distance in $ \mathbb{R}^{N}$, m the Lebesgue measure in $\mathbb{R}^{N}$ or $ \mathbb{R}^{N-1}$.

A mesh of $ \Omega$ is given by a family $ \M$ of control volumes, which are open polyhedral subsets of $\Omega$,  a family $ \mathcal{E}$ of relatively open parts of hyperplanes in $ \mathbb{R}^{N}$, which represent the faces of the control volumes,  and a family of points $(x_{K})_{K \in \M}$ which satisfy the following properties:
\begin{itemize}
\item $ \ds{\overline{\Omega}=\bigcup_{K \in \M}\overline{K}}$,
\item for all $K \in \M$, there exists $\E_{K} \subset \E$ such that $\ds{\pa K=\bigcup_{\sigma \in \E_{K}}\overline{\sigma}}$,
\item for all $(K,L) \in \M^{2}$ with $K \neq L$, either $\text{m}(\overline{K}\cap \overline{L})=0$, or $\overline{K} \cap \overline{L} = \overline{\sigma}$ for some $\sigma \in \E$, which will be denoted by $K|L$,
\item the family of points $(x_{K})_{K \in \T}$ is such that for all $K \in \M$, $x_{K} \in K$ and if $\sigma=K|L$, it is assumed that $x_{K} \neq x_{L}$.
\end{itemize}
In the set of faces $ \mathcal{E}$, we distinguish the interior faces $ \sigma \in \mathcal{E}_{int}$ and the boundary faces $ \sigma \in \mathcal{E}_{ext}$. For a control volume $K \in \M$, we denote by $ \mathcal{E}_{K}$ the set of its faces, $ \mathcal{E}_{int,K}$ the set of its interior faces and $ \mathcal{E}_{ext,K}$ the set of faces of $K$ included in the boundary $ \Gamma$.\\
For all $ \sigma \in \mathcal{E}$, we define 
\begin{equation*}
d_{\sigma}= \left\{\begin{array}{lcl} \text{d}(x_{K},x_{L}) & \text{ for } & \sigma=K|L \in \mathcal{E}_{int}, \\ \text{d}(x_{K},\sigma) & \text{ for } & \sigma \in \mathcal{E}_{ext,K}.\end{array}\right.
\end{equation*}
We assume that the mesh satisfies the following regularity constraint: there exists $ \xi >0$ such that 
\begin{equation}
\text{d}(x_{K},\sigma) \geq \xi \, d_{\sigma}, \quad \forall  K \in \M, \ \forall  \sigma \in \mathcal{E}_{K}.
\label{regmesh}
\end{equation}
The size of the mesh is defined by
\begin{equation*}
h = \max_{K \in \M}\left(\text{diam}(K)\right).
\end{equation*}
\\
In general, finite volume methods lead to the computation of one discrete unknown by control volume. The corresponding finite volume approximation is a piecewise constant function. Therefore, we define the set $X(\M)$ of the finite volume approximation: 
\begin{equation*}
X(\M)=\left\{u \in L^{1}(\Omega)\ /\ \exists (u_K)_{K\in\M}\mbox{ such that } u = \sum_{K \in \M}u_{K}\mathbf{1}_{K}\right\}.
\end{equation*}
Let us now define some discrete norms and seminorms on $X(\M)$. Firstly we consider the classical $L^{p}$ norm for piecewise constant functions: for $p\in [1,+\infty)$, 
$$
\Vert \, u \, \Vert_{0,p} \, = \, \left(\int_{\Omega}|u(x)|^{p}\,dx\right)^{\frac{1}{p}}\,=\, \left(\sum_{K \in \M}\text{m}(K)\, |\,u_{K}\,|^{p}\right)^{\frac{1}{p}}, \quad \forall u\in X(\M).
$$
\begin{defi}
Let $ \Omega$ be a bounded polyhedral subset of $ \mathbb{R}^{N}$, $ \M$ a  mesh of $\Omega$.
\begin{enumerate}
\item In the general case, for $p\in[1,+\infty)$, the discrete $W^{1,p}$-seminorm is defined by:
$$
|\,u\,|_{1,p,\M} \, = \, \left(\mathop{\sum_{\sigma \in \mathcal{E}_{int}}}_{\sigma=K|L}\frac{\text{m}(\sigma)}{d_{\sigma}^{p-1}}\, \left\vert u_L-u_K\right\vert^{p}\right)^{\frac{1}{p}}, \quad \forall u\in X(\M)
$$
and the discrete $W^{1,p}$-norm is defined by
\begin{equation}\label{normW1p.gc}
\Vert \, u  \, \Vert_{1,p,\M} \, = \, \Vert \, u \, \Vert_{0,p} \, + \, |\,u\,|_{1,p,\M}, \quad \forall u\in X(\M).
\end{equation}
\item In the case where homogeneous Dirichlet boundary conditions are underlying (because the piecewise constant function comes from a finite volume scheme), we need to take into account jumps on the boundary in the discrete $W^{1,p}$-seminorm. Let $\Gamma^{0} \subset \Gamma$ be a part of the boundary. In the set of exterior faces $ \mathcal{E}_{ext}$, we distinguish $ \mathcal{E}_{ext}^{0}$ the set of boundary faces included in $\Gamma^{0}$. For $p\in [1,+\infty)$, we define the discrete $W^{1,p}$-seminorm (which depends on $ \Gamma^{0}$)  by
\begin{equation}\label{seminormW1p.hc}
|\,u\,|_{1,p,\Gamma^0,\M} \, = \, \left(\sum_{\sigma \in \mathcal{E}}\frac{\text{m}(\sigma)}{d_{\sigma}^{p-1}}\, \left(D_{\sigma}u\right)^{p}\right)^{\frac{1}{p}}, \quad 1 \leq p < +\infty,
\end{equation}
where
\begin{equation*}
D_{\sigma}u = \left\{\begin{array}{lcl} \left| \, u_{K}-u_{L} \, \right| & \text{ if } & \sigma=K|L \in \mathcal{E}_{int}, \\ \left| \, u_{K} \, \right| & \text{ if } & \sigma \in \mathcal{E}_{ext}^{0} \cap \mathcal{E}_{K} ,\\  0  & \text{ if } & \sigma \in \mathcal{E}_{ext}\setminus \mathcal{E}_{ext}^{0}. \end{array}\right.
\end{equation*}
We then define the discrete $W^{1,p}$ norm by
\begin{equation}\label{normW1p.hc}
\Vert \, u  \, \Vert_{1,p,\M} \, = \, \Vert \, u \, \Vert_{0,p} \, + \, |\,u\,|_{1,p,\Gamma^0,\M},  \ \forall u\in X(\M).
\end{equation}
\end{enumerate}
\end{defi}

\begin{rem}
Afterwards we may often use the following inequalities:
\begin{align}
&\Vert\,u\,\Vert_{0,s}\,\leq \, \text{m}(\Omega)^{(p-s)/(ps)}\,\Vert\,u\,\Vert_{0,p},\ \forall 1\leq s\leq p,\label{holder0}\\
&|\,u\,|_{1,s,\M} \leq \left(\frac{N\,\text{m}(\Omega)}{\xi}\right)^{(p-s)/(ps)}\,|\,u\,|_{1,p,\M},\ \forall 1\leq s\leq p. \label{holder1}
\end{align} 
These inequalities result from Hölder's inequality applied with exponent $p/s \geq 1$. Let us detail the proof of \eqref{holder1}. Using Hölder's inequality we get: 
\begin{align*}
|\,u\,|_{1,s,\M}^{s}&=\sum_{\substack{\sigma\in \E_{int}\\ \sigma=K|L}}(\text{m}(\sigma)\text{d}_\sigma)^{s/p}\,\left(\frac{|u_{L}-u_{K}|}{d_{\sigma}}\right)^s(\text{m}(\sigma)\,d_{\sigma})^{(p-s)/p}\\
& \leq |\,u\,|_{1,p,\M}^{s}\times \left(\sum_{\sigma \in \E_{int}}\text{m}(\sigma)\,d_{\sigma}\right)^{(p-s)/p}.
\end{align*}
But  the regularity constraint \eqref{regmesh} on the mesh ensures that
\begin{equation}\label{cons_reg}
\sum_{\sigma \in \E_{int}}\text{m}(\sigma) \,d_{\sigma} \, \leq \,\frac{1}{\xi}\sum_{K \in \T}\sum_{\sigma \in \E_{K}}\text{m}(\sigma)\,\text{d}(x_{K},\sigma)=\frac{N}{\xi}\sum_{K \in \T}\text{m}(K)=\frac{N\,\text{m}(\Omega)}{\xi},
\end{equation}
which finally yields the result.
\end{rem}

\subsection{The space $BV(\Omega)$}\label{sec.BV}

Let us first recall some results concerning functions of bounded variation (we refer to \cite{Ambrosio2000,Ziemer1989} for a thorough presentation $BV(\Omega)$). Let $ \Omega$ be an open set of $ \mathbb{R}^{N}$ and $u \in L^{1}(\Omega)$. The \textit{total variation} of $u$ in $\Omega$, denoted by $TV_{\Omega}(u)$, is defined by
\begin{equation*}
TV_{\Omega}(u)=\sup \left\{ \int_{\Omega}u(x)\, \text{div}\left(\phi(x)\right)\, dx, \quad \phi \in \mathcal{C}^{1}_{c}(\Omega),\quad |\phi(x)|\leq 1, \quad \forall x \in \Omega \right\}
\end{equation*}
and the function $u \in L^{1}(\Omega)$ belongs to $BV(\Omega)$ if and only if $TV_{\Omega}(u)<+\infty$. The space $BV(\Omega)$ is endowed with the norm
\begin{equation*}
\Vert\, u \,\Vert_{BV(\Omega)} := \Vert\, u \,\Vert_{L^{1}(\Omega)}+TV_{\Omega}(u).
\end{equation*}
The space $BV(\Omega)$ is a natural space to study finite volume approximations. Indeed,  for $u\in X(\M)$, we have 
\begin{equation}
TV_{\Omega}(u)= \mathop{\sum_{\sigma \in \mathcal{E}_{int}}}_{\sigma=K|L}\text{m}(\sigma)\, \left\vert u_L-u_K\right\vert=|u|_{1,1,\M}<+\infty.
\label{relBVW11}
\end{equation}
The discrete space $X(\M)$ is included in $L^1\cap BV(\Omega)$.
 Moreover, $\Vert u\Vert_{BV(\Omega)}= \Vert u\Vert_{1,1,\M}$. 
 
Our starting point for the discrete functional inequalities is the continuous embedding of $BV(\Omega)$ into $L^{N/(N-1)}(\Omega)$ for a Lipschitz domain $\Omega$, recalled in Theorem \ref{thmBV1}. More details about the following results can be found in \cite[Chapter 3]{Ambrosio2000} and \cite[Chapter 5]{Ziemer1989}.
\begin{thrm}
Let $ \Omega$ be a Lipschitz bounded domain of $ \mathbb{R}^{N}$, $N\geq 2$. Then there exists a constant $c(\Omega)$ only depending on $\Omega$ such that:
\begin{equation}
\left(\int_{\Omega}|\,u\,|^{\frac{N}{N-1}}\, dx \right)^{\frac{N-1}{N}} \, \leq \, c(\Omega) \,\, \Vert\, u \,\Vert_{BV(\Omega)}, \quad \forall  u \in BV(\Omega).
\label{BV1}
\end{equation}
\label{thmBV1}
\end{thrm}
Moreover, if $\Omega$ is a connected domain, there are also more precise results involving only the seminorm $TV_{\Omega}(u)$ instead of the norm $ \Vert\, u \,\Vert_{BV(\Omega)}$. Indeed, the seminorm $TV_{\Omega}$ becomes a norm on the space of BV functions vanishing on a part of the boundary and also on the space of BV functions with a zero mean value. In these cases, the continuous embedding of $BV(\Omega)$ into $L^{N/(N-1)}(\Omega)$ is rewritten as in Theorem \ref{thmBV2}.
\begin{thrm}
Let $ \Omega$ be a Lipschitz bounded connected domain of $ \mathbb{R}^{N}$, $N\geq 2$. 
\begin{enumerate}
\item There exists a constant $c(\Omega)>0$ only depending on $ \Omega$ such that, for all $u\in BV(\Omega)$,
\begin{equation}
\left(\int_{\Omega}|\,u-\overline{u}\,|^{\frac{N}{N-1}}\, dx \right)^{\frac{N-1}{N}} \, \leq \, c(\Omega)\,\, TV_{\Omega}(u),
\label{BV3}
\end{equation} 
where $\overline{u}$ is the mean value of $u$:
$$
\overline{u} \, = \, \frac{1}{\text{m}(\Omega)}\, \int_{\Omega}u(x) \, dx.
$$
\item Let $\Gamma^{0} \subset \pa \Omega$, $\text{m}(\Gamma^{0}) >0$. There exists a constant $c(\Omega)>0$ only depending on $ \Omega$ and $\Gamma^{0}$ such that, for all $u\in BV(\Omega)$ satisfying $u=0$ on $ \Gamma^{0}$,
\begin{equation}
\left(\int_{\Omega}|\,u\,|^{\frac{N}{N-1}}\, dx \right)^{\frac{N-1}{N}} \, \leq \, c(\Omega) \,\, TV_{\Omega}(u).
\label{BV2}
\end{equation}

\end{enumerate}
\label{thmBV2}
\end{thrm}


\section{Discrete functional inequalities in the general case}\label{sec.gc}

We first consider the general case $u\in X(\M)$ with the discrete $W^{1,p}$ norm defined by \eqref{normW1p.gc}. The discrete functional inequalities we will prove may be useful in the convergence analysis of finite volume methods for problems with homogeneous Neumann boundary conditions. 

\subsection{General discrete Poincaré-Sobolev inequality}

We start with a Lemma which will be useful to prove the discrete Poincaré-Sobolev inequalities which are the discrete counterpart of~\eqref{sobolevcontinu} and then the discrete Gagliardo-Nirenberg-Sobolev inequalities.
 
\begin{lem}
\label{lmm00}
Let $ \Omega$ be an open bounded polyhedral domain of $ \mathbb{R}^{N}$, $N\geq 2$ and  $ \M$ be a mesh satisfying (\ref{regmesh}). For all $s> 1$ and $p>1$, we have
\begin{equation}
\Vert\,u\,\Vert_{0,sN/(N-1)}^{s}\,\leq \, \frac{C}{\xi^{(p-1)/p}}\Vert\,u\,\Vert^{(s-1)}_{0,(s-1)p/(p-1)}\Vert\,u\,\Vert_{1,p,\M},
\label{ineg3ps}
\end{equation}
\end{lem}
\begin{proof}
We prove this lemma by following the same kind of computations as in \cite{Bouchut2011,Chainais-Hillairet2011,Coudiere2001,Droniou2003,Eymard2000}.\\
Let $s > 1$ and $u \in X(\M)$. We apply the inequality \eqref{BV1} to $v \in X(\M)$ given by $v_{K}=|u_{K}|^{s}$ for all $K \in \M$. We obtain that
\begin{equation*}
\Vert\, u \,\Vert_{0,sN/(N-1)}^{s}\, \leq \,c(\Omega)\,\left(\sum_{\substack{\sigma \in \E_{int}\\ \sigma=K|L}}\text{m}(\sigma)\,\left|\vphantom{u^{L+1}} |u_{K}|^{s}-|u_{L}|^{s}\right|+\Vert\, u\,\Vert_{0,s}^{s}\right).
\end{equation*}
However, for all $\sigma=K|L$ and $s > 1$, we have
\begin{equation*}
\left|\vphantom{u^{L+1}} |u_{K}|^{s}-|u_{L}|^{s}\right|\leq s\left(|u_{K}|^{s-1}+|u_{L}|^{s-1}\right)|u_{K}-u_{L}|.
\end{equation*}
Applying this last inequality, a discrete integration by parts and Hölder's inequality, we get for any $p>1$ and $s > 1$:
\begin{multline*}
\sum_{\substack{\sigma \in \E_{int}\\ \sigma=K|L}}\text{m}(\sigma)\left|\vphantom{u^{L+1}} |u_{K}|^{s}-|u_{L}|^{s}\right| \leq  s\left(\sum_{\substack{\sigma \in \E_{int}\\ \sigma=K|L}}\frac{\text{m}(\sigma)}{d_{\sigma}^{p-1}}|u_{L}-u_{K}|^{p}\right)^{\frac{1}{p}}\times\\\left(\sum_{K \in \M}\sum_{\sigma \in \E_{K}}\text{m}(\sigma)\,d_{\sigma}\,|u_{K}|^{(s-1)p/(p-1)}\right)^{\frac{p-1}{p}}.
\end{multline*}
Thanks to \eqref{cons_reg}, it yields, for any $p>1$, $s> 1$,
\begin{equation}
\Vert\,u\,\Vert_{0,sN/(N-1)}^{s}\,\leq \, C\left(\frac{1}{\xi^{(p-1)/p}}\,|\,u\,|_{1,p,\M}\Vert\,u\,\Vert^{s-1}_{0,(s-1)p/(p-1)}+\Vert\,u\,\Vert_{0,s}^{s}\right).
\label{ineg1ps}
\end{equation}
On the one hand, by interpolation between $L^{p}$ spaces, since
\begin{equation*}
\frac{1}{s}=\frac{1/s}{p}+\frac{(s-1)/s}{(s-1)p/(p-1)} \,<\, 1,
\end{equation*}
we obtain for any $s\geq 1$ such that $(s-1)p/(p-1)\geq 1$:
\begin{equation}
\Vert\,u\,\Vert_{0,s}\,\leq\,\Vert\,u\,\Vert_{0,p}^{1/s}\,\Vert\,u\,\Vert^{(s-1)/s}_{0,(s-1)p/(p-1)}.
\label{ineg2ps}
\end{equation}
Using \eqref{ineg2ps},  \eqref{ineg1ps} is rewritten 
\begin{equation*}
\Vert\,u\,\Vert_{0,sN/(N-1)}^{s}\,\leq \, \frac{C}{\xi^{(p-1)/p}}\Vert\,u\,\Vert^{(s-1)}_{0,(s-1)p/(p-1)}\Vert\,u\,\Vert_{1,p,\M},
\end{equation*}
which is the expected result.
\end{proof}
Now let us prove the discrete analog to (\ref{sobolevcontinu}), which is already partially proven in \cite[Lemma~A.1]{Chainais-Hillairet2011}.

\begin{thrm}[General discrete Poincaré-Sobolev inequality]
Let $ \Omega$ be an open bounded polyhedral domain of $ \mathbb{R}^{N}$, $N\geq 2$. Let $ \M$ be a mesh satisfying (\ref{regmesh}).
\begin{itemize}
\item If $1\leq p<N$, let $1\leq q\leq p^{*}=\ds\frac{pN}{N-p}$,
\item  if $p\geq N$, let $1\leq q <+\infty$. 
\end{itemize}Then there exists a constant $C>0$ only depending on $p$, $q$, $N$ and $\Omega$ such that:
\begin{equation}
\Vert\, u \,\Vert_{0,q} \, \leq \, \frac{C}{\xi^{(p-1)/p}}\,\, \Vert\, u \,\Vert_{1,p,\M},\quad \forall u\in X(\M),
\label{poincare3}
\end{equation}
\label{thm4}
\end{thrm}

\begin{rem}
\label{sncf:0}
Let us emphasize that for $1\leq p < N$, this result allows to recover the optimal embedding with $q=p^\star$ as in the continuous case. However for $p > N$, we cannot prove the inequality with $q=\infty$ since our estimate is blowing-up in the limit $q\rightarrow \infty$.
\end{rem}

\begin{proof}
Throughout this proof, $C$ denotes constants which depend only on $\Omega$, $N$, $p$ and $q$. The proof is divided into four steps corresponding to different values of $p$: the case $p=1$, the case $1<p<N$, the critical case $p=N$ and finally the case $p> N$.

\paragraph*{Case $p=1$.}As seen in Section \ref{sec.BV}, we have $\Vert\,u\,\Vert_{BV(\Omega)}=\Vert\,u\,\Vert_{1,1,\M}$ for all $u\in X(\M)$. Therefore applying Theorem \ref{thmBV1}, we get 
\begin{equation*}
\Vert\,u\,\Vert_{0,N/(N-1)}\,\leq\,c(\Omega)\,\Vert\,u\,\Vert_{1,1,\M},
\end{equation*}
which is the result if $\ds{ q=p^{*}=N/(N-1) }$. It yields  (\ref{poincare3}) for $p=1$ and also for all $\ds{1 \leq q \leq p^{*} }$, thanks to \eqref{holder0}.

\paragraph*{Case $1<p<N$.} We start with the result (\ref{ineg3ps}) of Lemma~\ref{lmm00} and choose $s > 1$ such that $(s-1)p/(p-1)=sN/(N-1) > 1$, that is $s=(N-1)p/(N-p) > 1$, we have then $sN/(N-1)=Np/(N-p)$ and get
\begin{align*}
\Vert\,u\,\Vert_{0,pN/(N-p)}\, 
&\leq \, \frac{C}{\xi^{(p-1)/p}}\,\Vert\,u\,\Vert_{1,p,\M}.
\end{align*}
Finally since $L^{pN/(N-p)}(\Omega) \subset L^{q}(\Omega)$ for all $1 \leq q \leq pN/(N-p)$, we get that
\begin{equation*}
\Vert\,u\,\Vert_{0,q}\, \leq \,\frac{C}{\xi^{(p-1)/p}}\,\Vert\,u\,\Vert_{1,p,\M} \quad  1 \leq q \leq p^{*},
\end{equation*}
and the proof is complete for $1 \leq p <N$.

\paragraph*{Case $p=N$.}The proof is similar to that of the previous case but we cannot apply the last argument since $pN/(N-p)$ blows up for $p=N$. Therefore, we start  again from \eqref{ineg3ps} of Lemma~\ref{lmm00} for any $s > 1$ and $p=N$
\begin{equation*}
\Vert\,u\,\Vert_{0,sN/(N-1)}^{s}\,\leq\, \frac{C}{\xi^{(N-1)/N}}\,\Vert\,u\,\Vert_{1,N,\M}\,\Vert\,u\,\Vert_{0,(s-1)N/(N-1)}^{s-1}.
\end{equation*}
Now observing that $sN/(N-1)\geq (s-1)N/(N-1)$ and using inclusion of $L^{sN/(N-1)}(\Omega)$ in $L^{(s-1)N/(N-1)}(\Omega)$, we get that
\begin{equation*}
\Vert\,u\,\Vert_{0,(s-1)N/(N-1)}\,\leq\,\frac{C}{\xi^{(N-1)/N}}\,\Vert\,u\,\Vert_{1,N,\M} \quad \text{for } s > 1.
\end{equation*}
Then for all $q \geq 1$, we choose $s=1+(N-1)q/N > 1$, which yields that $q=(s-1)N/(N-1) \geq 1$ and finally we obtain the result:
\begin{equation}
\Vert\,u\,\Vert_{0,q}\,\leq\,\frac{C}{\xi^{(N-1)/N}}\,\Vert\,u\,\Vert_{1,N,\M} \quad \forall q \geq 1.
\label{ineg30ps}
\end{equation}
\paragraph*{Case $p>N$.}We obtain the result using the fact that 
\begin{equation}
\Vert\,u\,\Vert_{1,N,\M} \,\leq\,\frac{C}{\xi^{(p-N)/pN}}\,\Vert\,u\,\Vert_{1,p,\M} \quad \forall p \geq N.
\label{ineg4ps}
\end{equation}
Gathering \eqref{ineg30ps} and \eqref{ineg4ps} we get
\begin{equation*}
\Vert\,u\,\Vert_{0,q} \leq \frac{C}{\xi^{(N-1)/N}}\,\Vert\,u\,\Vert_{1,N,\M}\,\leq \,\frac{C}{\xi^{(p-1)/p}}\,\Vert\,u\,\Vert_{1,p,\M} \quad \forall q \geq 1,
\end{equation*}
which completes the proof of Theorem \ref{thm4}.
\end{proof}


\subsection{General discrete Gagliardo-Nirenberg-Sobolev inequality}

We now give the discrete counterpart of the Gagliardo-Nirenberg-Sobolev inequalities \eqref{GNScontinu}.

\begin{thrm}[General discrete Gagliardo-Nirenberg-Sobolev inequality]
Let $ \Omega$ be an open boun\-ded polyhedral domain of $ \mathbb{R}^{N}$, $N\geq 2$. Let $ \M$ be a mesh satisfying (\ref{regmesh}).
\begin{itemize}
\item If $1\leq p<N$, let $1\leq s \leq m \leq p^{\ast}=\ds\frac{pN}{(N-p)}$,
\item  if $p\geq N$, let $1\leq s \leq m <+\infty$. 
\end{itemize}
Then,  there exists a constant $C>0$ only depending on $p$, $s$, $m$, $N$ and $\Omega$ such that:
\begin{equation}\label{GNS}
\Vert\, u \,\Vert_{0,m} \, \leq \, \frac{C}{\xi^{(p-1)\theta /p}}\,\, \Vert\, u \,\Vert^{\theta}_{1,p,\M}\, \Vert \,u \,\Vert_{0,s}^{1-\theta},\ \forall u\in X(\M),
\end{equation}
where $\theta$ is given by \eqref{theta:opt}:
$$
\theta =  \frac{\ds\frac{1}{s}-\frac{1}{m}}{\ds\frac{1}{s}+\frac{1}{N}-\frac{1}{p}}.
$$
%
\label{thm3}
\end{thrm}

\begin{proof}
Throughout this proof, $C$ denotes constants which depend only on $\Omega$, $N$, $p$, $s$ and $\theta$. We distinguish the case $1\leq p<N$ from the case $p\geq N$.

\paragraph*{Case $1\leq p <N$.}
For all $1\leq s\leq m \leq p^{\ast}$, there exists $\theta \in [0,1]$  such that 
$$
\frac{1}{m}=\frac{\theta}{p^{\ast}}+\frac{1-\theta}{s}
$$
and $\theta$ is given by \eqref{theta:opt}. Interpolation between $L^p$ spaces gives:
\begin{equation}\label{inegGNS0}
\Vert\,u\,\Vert_{0,m}\,\leq\,\Vert\,u\,\Vert_{0,p^{\ast}}^{\theta}\,\Vert\,u\,\Vert_{0,s}^{1-\theta}.
\end{equation}
Applying Theorem \ref{thm4} with $q=p^{\ast}$, we deduce \eqref{GNS} from \eqref{inegGNS0}. 
\paragraph*{Case $p\geq N$. } We proceed by induction and first prove that \eqref{GNS} occurs in the case where 
\begin{equation}\label{condGNS}
1\leq s\leq m\leq s+1. 
\end{equation}
 Therefore, we start from \eqref{ineg3ps} in Lemma~\ref{lmm00} written with $r> 1$ instead of $s$ : 
$$
\Vert\,u\,\Vert_{0,rN/(N-1)}^{r}\,\leq \, \frac{C}{\xi^{(p-1)/p}}\Vert\,u\,\Vert^{(r-1)}_{0,(r-1)p/(p-1)}\Bigl(\,|\,u\,|_{1,p,\M}+ \Vert\,u\,\Vert_{0,p}\Bigl).
$$
With $r=1+(p-1)s/p$, which is equivalently written $s=(r-1)p/(p-1)$, we get 
\begin{equation}\label{inegGNS1}
\Vert\,u\,\Vert_{0,rN/(N-1)}\,\leq \, \frac{C}{\xi^{(p-1)/(pr)}}\Vert\,u\,\Vert^{(r-1)/r}_{0,s} \Vert\,u\,\Vert_{1,p,\M}^{1/r}.
\end{equation}
Now since $rN/(N-1)\geq s+1$ for $p\geq N\geq 2$, we have for any $m$  verifying \eqref{condGNS}, there exists $\alpha\in[0,1]$ such that 
\begin{equation}\label{defalpha}
\frac{1}{m}= \frac{\alpha}{rN/(N-1)}+\frac{1-\alpha}{s},
\end{equation}
which implies by interpolation between $L^p$ spaces,
$$
\Vert u\Vert_{0,m}\leq \Vert u\Vert_{0,rN/(N-1)}^\alpha\Vert u\Vert_{0,s}^{1-\alpha}
$$
and using \eqref{inegGNS1}, we get:
\begin{equation}\label{inegGNS2}
\Vert u\Vert_{0,m}\leq \frac{C}{\xi^{\alpha(p-1)/(pr)}}\Vert\,u\,\Vert^{1-\alpha/r}_{0,s} \Vert\,u\,\Vert_{1,p,\M}^{\alpha/r}.
\end{equation}
It remains to verify that \eqref{defalpha} with $r=1+(p-1)s/p$ implies that 
$$
\frac{\alpha}{r}= \ds\frac{\ds\frac{1}{s}-\frac{1}{m}}{\ds\frac{1}{s}+\frac{1}{N}-\frac{1}{p}}=\theta.
$$
As $\alpha/r=\theta$, inequality \eqref{inegGNS2} corresponds to the expected inequality \eqref{GNS}. The result is proved if $m$ satisfies \eqref{condGNS}.  


Then, let us now prove by induction that it holds for $1\leq s\leq m<+\infty$. 

For a given $k\in \R$, we assume that \eqref{GNS} is satisfied for $1\leq s+k\leq m\leq s+k+1$. Then, for $m=s+k+1$, we have 
\begin{equation}\label{inegGNS10}
\Vert u\Vert_{0,s+k+1} \leq \frac{C}{\xi^{(p-1){\bar \theta}/p}}\Vert u\Vert_{1,p,\M}^{\bar \theta} \Vert u\Vert_{0,s}^{1-{\bar\theta}},
\end{equation}
with 
$$
{\bar \theta}=\ds\frac{\ds\frac{1}{s}-\frac{1}{s+k+1}}{\ds\frac{1}{s}+\frac{1}{N}-\frac{1}{p}}.
$$
Let us now choose $1\leq s+k+1\leq m\leq s+k+2$, applying \eqref{GNS} with $s+k+1$ instead of $s$, we have 
\begin{equation}\label{inegGNS11}
 \Vert u\Vert_{0,m} \leq \frac{C}{\xi^{(p-1){\tilde \theta}/p}}\Vert u\Vert_{1,p,\M}^{\tilde \theta} \Vert u\Vert_{0,s+k+1}^{1-{\tilde\theta}}
\end{equation}
with 
$$
{\tilde \theta}=\ds\frac{\ds\frac{1}{s+k+1}-\frac{1}{m}}{\ds\frac{1}{s+k+1}+\frac{1}{N}-\frac{1}{p}}.
$$
Injecting \eqref{inegGNS10} in \eqref{inegGNS11}, we get :
$$
\Vert u\Vert_{0,m} \leq \frac{C}{\xi^{(p-1)({\tilde \theta}+\bar\theta(1-\tilde \theta)/p}}\Vert u\Vert_{1,p,\M}^{\tilde \theta+\bar\theta(1-\tilde\theta)} \Vert u\Vert_{0,s}^{(1-\bar\theta)(1-{\tilde\theta})}.
$$
But $(1-\bar\theta)(1-\tilde \theta)=(1-\theta)$  and $\tilde \theta+\bar\theta(1-\tilde\theta)=\theta$, with $\theta$ given by \eqref{theta:opt}. It means that \eqref{GNS} holds for $1\leq s+k+1\leq m\leq s+k+2$ and it concludes the proof by induction.
\end{proof}


\subsection{Other discrete functional inequalities}

From Theorems \ref{thm4} and \ref{thm3}, we can deduce a discrete Nash inequality, which can be used in the proof of existence of solution for stochastic PDEs including non-monotone stochastic generalized porous-medium equations \cite{Nash2} or to get lower bound of solutions of congestion games \cite{Nash1}:

\begin{cor}[Discrete Nash inequality]
Let $ \Omega$ be an open bounded polyhedral domain of $ \mathbb{R}^{N}$. Let $ \M$ be a mesh satisfying (\ref{regmesh}).
Then there exists a constant $C>0$ only depending on $ \Omega$ and $N$ such that
\begin{equation*}
\Vert\, u \,\Vert_{0,2}^{1+\frac{2}{N}} \, \leq \, \frac{C}{\sqrt{\xi}} \,\, \Vert\, u \,\Vert_{1,2,\M} \, \Vert\, u \,\Vert^{\frac{2}{N}}_{0,1},\quad \forall u\in X(\M).
\end{equation*}
\end{cor}

\begin{proof}
For $N=2$, the result is directly given by the application of Theorem \ref{thm3} with $p=2$, $s=1$, $\theta=1/N=1/2$ and $m=2$. 
For $N\geq 3$, let us first apply H\"older's inequality:
\begin{equation}
\Vert\, u \,\Vert_{0,2}^{2} = \sum_{K \in \M}\text{m}(K) |u_{K}|^{4/(N+2)}\, |u_{K}|^{2N/(N+2)} \, \leq \, \Vert\, u \,\Vert_{0,1}^{4/(N+2)} \, \Vert\, u \,\Vert_{0,2N/(N-2)}^{2N/(N+2)}.
\label{nashp1}
\end{equation}
Then we apply Theorem \ref{thm4} with $1 \leq p=2<N$ and $q=p^{*}=2N/(N-2)$:
\begin{equation}
\Vert\, u \,\Vert_{0,2N/(N-2)} \, \leq \, \frac{C}{\sqrt{\xi}}\, \Vert\, u \,\Vert_{1,2,\M}.
\label{nashp2}
\end{equation}
Gathering (\ref{nashp1}) and (\ref{nashp2}), it yields the result.

\end{proof}

In the proofs of Theorem \ref{thm4} and Theorem \ref{thm3}, we have used the continuous embedding of $BV(\Omega)$ into $L^{N/(N-1)}(\Omega)$ as it is written in Theorem \ref{thmBV1}. But, starting with \eqref{BV3} instead of \eqref{BV1} leads to the following discrete Poincar\'e-Wirtinger inequality. 

\begin{thrm}[Discrete Poincaré-Wirtinger inequality]
Let $ \Omega$ be an open bounded connected polyhedral domain of $ \mathbb{R}^{N}$. Let $ \M$ be a mesh satisfying (\ref{regmesh}). Then for $1 \leq p < + \infty$ there exists a constant $C>0$ only depending on $ \Omega$, $N$ and $p$ such that:
\begin{equation}
\left\Vert\, u-\overline{u}\,\right\Vert_{0,p} \leq \frac{C}{\xi^{(p-1)/p}}\left|\, u \,\right|_{1,p,\M} \quad \forall u \in X(\M).
\label{poincarewirtinger}
\end{equation}
We recall that ${\bar u}=\ds\frac{1}{\text{m}(\Omega)}\int_{\Omega} u(x) dx= \ds\frac{1}{\text{m}(\Omega)}\sum_{K\in{\M}} \text{m}(K)u_K$, for $u \in X({\M})$.
\end{thrm}

\begin{proof}
Throughout this proof, $C$ denotes constants which depend only on $ \Omega$, $N$ and $p$.\\
Using inequalities \eqref{holder0} and \eqref{BV3}, we have for all $u \in X(\M)$:
\begin{equation*}
\Vert\, u-\overline{u}\,\Vert_{0,p} \leq C \, \Vert\,u-\overline{u}\,\Vert_{0,N/(N-1)}\leq C \, |\,u\,|_{1,1,\M} \quad \forall 1 \leq p \leq \frac{N}{N-1}.
\end{equation*}
Then applying inequality \eqref{holder1}, it gives the result \eqref{poincarewirtinger} for $1 \leq p \leq N/(N-1)$:
\begin{equation}
\Vert\, u-\overline{u}\,\Vert_{0,p} \leq \frac{C}{\xi^{(p-1)/p}}\,|\,u\,|_{1,p,\M} \quad \forall 1 \leq p \leq \frac{N}{N-1}.
\label{ineg0pw}
\end{equation}
Let us now take $s \geq 1$. For $ u \in X(\M)$, we define $v \in X(\M)$ by $v_{K}=|\,u_{K}-\overline{u}\,|^{s}$ for all $K \in \M$. On the one hand, we have
\begin{align}
\left\Vert\, v - \overline{v} \,\right\Vert_{0,N/(N-1)} & \geq \Vert\, v \,\Vert_{0,N/(N-1)}-\Vert\, \overline{v} \,\Vert_{0,N/(N-1)} \nonumber \\\,\nonumber\\
& \geq \Vert\, u-\overline{u}\,\Vert_{0,sN/(N-1)}^{s}-\frac{\Vert\, u-\overline{u} \,\Vert_{0,s}^{s}}{\text{m}(\Omega)^{1/N}}.\label{ineg1pw}
\end{align}
On the other hand, using the same technique as in the proof of Theorem \ref{thm4}, we have:
\begin{equation}
|\,v\,|_{1,1,\M} \leq \frac{C\,s}{\xi^{(p-1)/p}}\,|\,u\,|_{1,p,\M}\,\Vert\, u-\overline{u}\,\Vert_{0,(s-1)p/(p-1)}^{(s-1)} \quad \forall p > 1, \quad \forall s \geq 1.
\label{ineg2pw}
\end{equation}
Therefore, gathering \eqref{ineg1pw} and \eqref{ineg2pw}, we get from \eqref{BV3} that for $u \in X(\M)$, $p > 1$ and $s \geq 1$:
\begin{equation*}
\Vert\, u-\overline{u}\,\Vert_{0,sN/(N-1)}^{s} \leq \frac{Cs}{\xi^{(p-1)/p}} \, |\,u\,|_{1,p,\M}\,\Vert\,u-\overline{u}\,\Vert_{0,(s-1)p/(p-1)}^{(s-1)}+\frac{\Vert\, u-\overline{u} \,\Vert_{0,s}^{s}}{\text{m}(\Omega)^{1/N}}.
\end{equation*}
Moreover, by interpolation between $L^{r}$ spaces, for $q \geq 1$, $p > 1$, $s \geq 1$ such that 
\begin{equation*}
\frac{1}{s}=\frac{1/s}{p}+\frac{(s-1)/s}{q}
\end{equation*}
we have that
\begin{equation*}
\Vert\,u-\overline{u}\,\Vert_{0,s} \leq \Vert\, u-\overline{u}\,\Vert_{0,p}^{1/s}\,\Vert\,u-\overline{u}\,\Vert_{0,q}^{(s-1)/s}.
\end{equation*}
Choosing $q=sN/(N-1)\geq 1$ and $s=(N-1)p/(N-p)$ for $1< p<N$, it yields
\begin{equation}
\Vert\,u-\overline{u}\,\Vert_{0,q} \leq C \left(\frac{1}{\xi^{(p-1)/p}}\,|\,u\,|_{1,p,\M}+\Vert\,u-\overline{u}\,\Vert_{0,p} \right) \quad \forall 1 \leq p<N,
\label{ineg3pw}
\end{equation}
where
\begin{equation*}
q=q(p):=\frac{pN}{N-p}.
\end{equation*}
Applying \eqref{ineg0pw} and using the fact that since $p<q=pN/(N-p)$, inequality \eqref{holder1} provides
\begin{equation}
|\,u\,|_{1,p,\M} \leq \frac{C}{\xi^{(q-p)/pq}}\,|\,u\,|_{1,q,\M}=\frac{C}{\xi^{1/N}}\,|\,u\,|_{1,q,\M},
\label{holder}
\end{equation}
we can estimate the right hand side of \eqref{ineg3pw} for $1\leq p \leq N/(N-1)$, which yields:
\begin{equation*}
\Vert\, u-\overline{u}\,\Vert_{0,q} \leq \frac{C}{\xi^{(q-1)/q}}\,|\,u\,|_{1,q,\M} \quad \forall 1 \leq q \leq \frac{N}{N-2}.
\end{equation*}
Using exactly the same technique, we proceed by induction to prove the result for $1 \leq p \leq N/(N-k)$, up to $k=N-1$, which finally yields the result for all $1 \leq p \leq N$.\\
To conclude, we apply \eqref{ineg3pw}. Indeed, using the result for $1 \leq p <N$ and the inequality \eqref{holder}, since $q=q(p)=pN/(N-p) \in [1;+\infty)$ if $p \in [1;N)$, we obtain the general result:
\begin{equation*}
\Vert\, u-\overline{u}\,\Vert_{0,q} \leq \frac{C}{\xi^{(q-1)/q}}\,|\,u\,|_{1,q,\M} \quad \forall 1 \leq q <+\infty.
\end{equation*}
\end{proof}


\section{Discrete functional inequalities in the case of Dirichlet boundary conditions}\label{sec.dirichlet}

In this section, we consider the case where the finite volume approximation $u\in X({\M})$ is coming from a finite volume scheme where homogeneous boundary conditions are prescribed on a part of the boundary. This part of the boundary is denoted by $\Gamma^0\subset \Gamma$, $\text{m}(\Gamma^0)>0$. In this case, the natural discrete counterparts of the $W^{1,p}$seminorm and $W^{1,p}$ norm are defined by \eqref{seminormW1p.hc} and \eqref{normW1p.hc}.
Moreover, the $W^{1,p}$ seminorm becomes a norm on the space of $W^{1,p}$ functions vanishing on a part of the boundary and the Gagliardo-Nirenberg-Sobolev inequalities and the Poincar\'e-Sobolev inequalities may be rewritten with the $W^{1,p}$-seminorm instead of the $W^{1,p}$-norm. Our aim in this Section is to prove the discrete counterpart of such inequalities (see Theorem \ref{thm2} and Theorem~\ref{thm1}).

As in the general case, the starting point will be the continuous embedding from $BV(\Omega)$ into $L^{N/(N-1)}(\Omega)$, which is rewritten as \eqref{BV2} with homogeneous Dirichlet boundary conditions on the part of the boundary. However, \eqref{BV2} can not be directly applied to $u\in X({\M})$. Indeed, $u\in X({\M})$ belongs to $BV(\Omega)$ and therefore its trace on the boundary is well defined; but it does not necessarily vanish on $\Gamma^0$. Some adaptations must be done in order to apply \eqref{BV2} and get its discrete counterpart. It will be done in Section \ref{sec_prellim} and yield the discrete functional inequalities presented in Section \ref{sec_Sobolevdir} and Section \ref{sec_GNSdir}.

In this section, we assume that the open set $\Omega$ is also connected to apply the result \eqref{BV2} of Theorem \ref{thmBV2}.

\subsection{Preliminary Lemma}\label{sec_prellim}

We begin with a lemma which gives the discrete counterpart of \eqref{BV2}. This lemma is crucial to prove Theorems \ref{thm2} and \ref{thm1}.

\begin{lem}
Let $ \Omega$ be an open connected bounded polyhedral domain of $ \mathbb{R}^{N}$ and  $\text{m}(\Gamma^{0}) >0$ be a part of the boundary $\Gamma$. Let $ \M$ be a mesh satisfying (\ref{regmesh}). 
Then there exists a constant $c(\Omega)$ only depending on $ \Omega$ and $\Gamma^{0}$ such that
\begin{equation*}
\Vert\, u \,\Vert_{0,N/(N-1)} \, \leq \, c(\Omega) \, |\,u\,|_{1,1,\Gamma^0,\M}, \quad \forall u\in X({\M}).
\end{equation*}
\label{lembase}
\end{lem}

\begin{proof}
Let us consider $u \in X(\M)$; since $u$ is piecewise constant, $u$ belongs to $BV(\Omega)$. Then we can define the trace $Tu$ of $u$ by: for almost every $x \in \Gamma$,
\begin{equation*}
\lim_{r \rightarrow 0}\,\frac{1}{\text{m}\left(B(x,r) \cap \Omega\right)} \int_{B(x,r) \cap \Omega} \left|u-Tu(x)\right|\, dy =0.
\end{equation*}
Thus in general $Tu_{|\Gamma^{0}} \neq 0$ and in this framework we cannot take into account some prescribed homogeneous Dirichlet boundary conditions $u_{\sigma}=0$ for $\sigma \in \mathcal{E}_{ext}^0$. Therefore the idea is to modify the mesh $\M$. For a parameter $\varepsilon>0$, we set $$ \Omega_{\varepsilon}=\{ x \in \Omega\,; \, \text{d}(x,\pa \Omega)>\varepsilon\}.$$ 
Then, for a given control volume $K\in\M$, two cases may happen: either $K\subset \Omega_{\varepsilon}$ or $K\cap (\Omega\setminus \Omega_{\varepsilon})\neq \emptyset$ (we assume that $\varepsilon$ is sufficiently small such that the case where $K\subset (\Omega\setminus\Omega_{\varepsilon})$ does not occur).  Then,
\begin{itemize} 
\item if $K\subset \Omega_{\varepsilon}$, we set $K_{\varepsilon}=K$,
\item if $K\cap (\Omega\setminus \Omega_{\varepsilon})\neq \emptyset$, we split the control volume $K$ into two control volumes $K_{\varepsilon}^1=K \cap \Omega_{\varepsilon}$ and $K_{\varepsilon}^2=K \cap (\Omega \setminus \Omega_{\varepsilon})$.
\end{itemize}
It defines a new set of control volumes, denoted by $\M_{\varepsilon}$ (see Figure \ref{fig1}). The corresponding set of edges is denoted by $\E_{\varepsilon}$. It contains the edges of $\E$ included in $\Omega_{\varepsilon}$, some edges of $\E$ crossing $\partial \Omega_{\varepsilon}$ and therefore split into two new edges and some new edges included in $\partial \Omega_{\varepsilon}$.

Let us now define a function $u_{\varepsilon}\in X(\M_{\varepsilon})$, which is still a piecewise constant function but which takes into account some boundary values $u_{\sigma}=0$ for $\sigma\subset \Gamma_0$: 
\begin{equation*}
u_{\varepsilon}=\sum_{K_{\varepsilon}\in \M_{\varepsilon}}u_{K_{\varepsilon}}\mathbf{1}_{K_{\varepsilon}},
\end{equation*}
where
\begin{equation*}
u_{K_{\varepsilon}}=\left\{\begin{array}{lcl}
u_{K} & \text{ if } & \text{m}(\pa K_{\varepsilon}\cap \Gamma^{0})=0 \text{ and } K_{\varepsilon} \subset K, \\ 0 & \text{ if } & \text{m}(\pa K_{\varepsilon}\cap \Gamma^{0})>0.
\end{array}\right.
\end{equation*}
This function verifies $u_{\varepsilon}=0$ on $\Gamma^0$ and we can apply \eqref{BV2}:
\begin{equation}\label{lemineg1}
\Vert u_{\varepsilon}\Vert_{0,N/(N-1)}\leq c(\Omega) TV_{\Omega}(u_{\varepsilon}).
\end{equation}
In order to pass to the limit $\varepsilon\to 0$ in this last inequality, we analyze the limit of both sides of the inequality. 
First, we note that 
 \begin{align*}
  \left\vert \Vert\, u_{\varepsilon}\,\Vert^{N/(N-1)}_{0,N/(N-1)}-\Vert\,u\,\Vert^{N/(N-1)}_{0,N/N-1}\right\vert&\leq\!\!\!\!\sum_{\substack{K_{\varepsilon}\in \M_{\varepsilon},\,K_{\varepsilon}\subset K \\ \text{m}(\pa K_{\varepsilon}\cap\Gamma^{0})>0}}\text{m}(K_{\varepsilon})\,|u_{K}|^{N/(N-1)}\\
  &\leq {\textrm m}(\Omega\setminus \Omega_{\varepsilon}) \Vert u\Vert_{0,\infty}.
    \end{align*}
    It implies that 
 \begin{equation}\label{lemlim1}
    \ds\lim_{\varepsilon\to 0}\Vert\, u_{\varepsilon}\,\Vert_{0,N/(N-1)}=\Vert\,u\,\Vert_{0,N/N-1}.
  \end{equation}  
Now it remains to compare $TV(u_{\varepsilon})$ and $TV(u)$ in the limit $\varepsilon \rightarrow 0$. We have
\begin{align*}
TV(u_{\varepsilon})&=\sum_{\substack{\sigma_{\varepsilon}=K_{\varepsilon}|L_{\varepsilon}\\ K_{\varepsilon},\,L_{\varepsilon} \subset \Omega_{\varepsilon}}}\text{m}(\sigma_{\varepsilon})\,|u_{L_{\varepsilon}}-u_{K_{\varepsilon}}|+\sum_{\substack{\sigma_{\varepsilon}=K_{\varepsilon}|L_{\varepsilon}\\ K_{\varepsilon}\subset \Omega_{\varepsilon},\,L_{\varepsilon} \subset \Omega \setminus\Omega_{\varepsilon}}}\text{m}(\sigma_{\varepsilon})\,|u_{L_{\varepsilon}}-u_{K_{\varepsilon}}|\\
&\quad +\sum_{\substack{\sigma_{\varepsilon}=K_{\varepsilon}|L_{\varepsilon}\\ K_{\varepsilon},\,L_{\varepsilon} \subset \Omega \setminus\Omega_{\varepsilon}}}\text{m}(\sigma_{\varepsilon})\,|u_{L_{\varepsilon}}-u_{K_{\varepsilon}}|.
\end{align*}  
The first term in $TV(u_{\varepsilon})$ tends to $TV(u)$ as $\varepsilon \rightarrow 0$, whereas the second term allows to take the Dirichlet boundary conditions into account:
$$ \sum_{\substack{\sigma_{\varepsilon}=K_{\varepsilon}|L_{\varepsilon}\\ K_{\varepsilon}\subset \Omega_{\varepsilon},\,L_{\varepsilon} \subset \Omega \setminus\Omega_{\varepsilon}}}\text{m}(\sigma_{\varepsilon})\,|u_{L_{\varepsilon}}-u_{K_{\varepsilon}}| \rightarrow \sum_{\sigma \in \E_{ext}^{0}\cap \E_{K}}\text{m}(\sigma)\,|u_{K}|\text{ as } \varepsilon \rightarrow 0.$$
For  the third term we have:
$$ \sum_{\substack{\sigma_{\varepsilon}=K_{\varepsilon}|L_{\varepsilon}\\ K_{\varepsilon},\,L_{\varepsilon} \subset \Omega \setminus\Omega_{\varepsilon}}}\text{m}(\sigma_{\varepsilon})\,|u_{L_{\varepsilon}}-u_{K_{\varepsilon}}| \leq 2\, \text{card}(\E)\, \varepsilon\,\|u\|_{0,\infty} \rightarrow 0 \text{ as } \varepsilon \rightarrow 0$$
and therefore:
\begin{equation}\label{lemlim2} 
\ds\lim_{\varepsilon\to 0} TV(u_{\varepsilon})= TV(u)+\sum_{\sigma \in \E_{ext}^{0}\cap \E_{K}}\text{m}(\sigma)|u_{K}|=|\,u\,|_{1,1,\Gamma^{0},\M}.
\end{equation}
Finally,  passing to the limit $ \varepsilon \rightarrow 0$ in \eqref{lemineg1} and using \eqref{lemlim1} and \eqref{lemlim2}, we obtain 
\begin{equation*}
\Vert\, u\,\Vert_{0,N/(N-1)}\, \leq \,c(\Omega)\, |\,u\,|_{1,1,\Gamma^{0},\M}.
\end{equation*}
\begin{figure}[ht!]
\centering
\subfigure{
\newrgbcolor{zzttqq}{0.6 0.2 0}
\newrgbcolor{fftttt}{1 0.2 0.2}
\psset{xunit=0.65cm,yunit=0.65cm,algebraic=true,dotstyle=o,dotsize=3pt 0,linewidth=0.8pt,arrowsize=3pt 2,arrowinset=0.25}
\begin{pspicture*}(-3,-5)(13,5)
\psline[linecolor=zzttqq](-1.58,4.04)(6.38,4.04)
\psline[linecolor=zzttqq](6.38,4.04)(12.84,-0.74)
\psline[linecolor=zzttqq](12.84,-0.74)(8.38,-4.62)
\psline[linewidth=4pt,linecolor=fftttt](8.38,-4.62)(4.8,-2.98)
\psline[linewidth=4pt,linecolor=fftttt](4.8,-2.98)(-0.66,-4.96)
\psline[linecolor=zzttqq](-0.66,-4.96)(-1.58,4.04)
\rput[tl](1.22,4.68){$\Omega$}
\rput[tl](3.26,-3.9){\fftttt{$\Gamma^0$}}
\psline(-0.66,-4.96)(1.76,-1.08)
\psline(-1.23,0.64)(1.76,-1.08)
\psline(4.8,-2.98)(1.76,-1.08)
\psline(4.8,-2.98)(6.98,-0.54)
\psline(6.98,-0.54)(10.58,-2.7)
\psline(6.98,-0.54)(9.69,1.59)
\psline(6.98,-0.54)(6.38,4.04)
\psline(1.76,-1.08)(6.98,-0.54)
\psline(1.76,-1.08)(6.38,4.04)
\psline(4.02,1.42)(1.7,4.04)
\psline(-1.23,0.64)(1.7,4.04)
\psline(6.98,-0.54)(8.38,-4.62)
\rput[tl](1.8,-2.28){$u_{K_2}$}
\rput[tl](4.42,-1.36){$u_{K_3}$}
\rput[tl](6.58,-2.52){$u_{K_4}$}
\rput[tl](-0.12,-1.46){$u_{K_1}$}
\rput[tl](8.58,-2.54){$u_{K_5}$}
\psline(4.02,1.42)(6.98,-0.54)
\end{pspicture*}}
\subfigure{
\newrgbcolor{fftttt}{1 0.2 0.2}
\newrgbcolor{zzttqq}{0.6 0.2 0}
\psset{xunit=0.65cm,yunit=0.65cm,algebraic=true,dotstyle=o,dotsize=3pt 0,linewidth=0.8pt,arrowsize=3pt 2,arrowinset=0.25}
\begin{pspicture*}(-3,-5)(13,5)
\pspolygon[linecolor=zzttqq,fillcolor=zzttqq,fillstyle=solid,opacity=0.1](-0.4,3.1)(0.3,-3.41)(4.66,-1.84)(7.99,-3.5)(11.2,-0.88)(5.98,2.92)
\psline(-1.58,4.04)(6.38,4.04)
\psline(6.38,4.04)(12.84,-0.74)
\psline(12.84,-0.74)(8.38,-4.62)
\psline[linewidth=4pt,linecolor=red](8.38,-4.62)(4.8,-2.98)
\psline[linewidth=4pt,linecolor=red](4.8,-2.98)(-0.66,-4.96)
\psline(-0.66,-4.96)(-1.58,4.04)
\rput[tl](1.22,4.68){$\Omega$}
\rput[tl](3.26,-3.9){\fftttt{$\Gamma^0$}}
\psline(-0.66,-4.96)(1.76,-1.08)
\psline(-1.23,0.64)(1.76,-1.08)
\psline(4.8,-2.98)(1.76,-1.08)
\psline(4.8,-2.98)(6.98,-0.54)
\psline(6.98,-0.54)(10.58,-2.7)
\psline(6.98,-0.54)(9.69,1.59)
\psline(6.98,-0.54)(6.38,4.04)
\psline(1.76,-1.08)(6.98,-0.54)
\psline(1.76,-1.08)(6.38,4.04)
\psline(4.02,1.42)(1.7,4.04)
\psline(-1.23,0.64)(1.7,4.04)
\psline(6.98,-0.54)(8.38,-4.62)
\rput[tl](10.9,-1.64){$\varepsilon$}
\rput[tl](5.3,2.36){\zzttqq{$\Omega_{\varepsilon}$}}
\rput[tl](-0.8,-1.38){$u_{K_1}$}
\rput[tl](0.3,-1.06){$u_{K_1}$}
\rput[tl](1.76,-2.1){$u_{K_2}$}
\rput[tl](1.98,-3.24){0}
\rput[tl](4.4,-1.14){$u_{K_3}$}
\rput[tl](4.3,-2.14){$u_{K_3}$}
\rput[tl](6.44,-1.98){$u_{K_4}$}
\rput[tl](6.52,-3.1){0}
\rput[tl](8.3,-2.12){$u_{K_5}$}
\rput[tl](8.92,-3){$u_{K_5}$}
\psline[linecolor=zzttqq](-0.4,3.1)(0.3,-3.41)
\psline[linecolor=zzttqq](0.3,-3.41)(4.66,-1.84)
\psline[linecolor=zzttqq](4.66,-1.84)(7.99,-3.5)
\psline[linecolor=zzttqq](7.99,-3.5)(11.2,-0.88)
\psline[linecolor=zzttqq](11.2,-0.88)(5.98,2.92)
\psline[linecolor=zzttqq](5.98,2.92)(-0.4,3.1)
\psline{->}(10.85,-1.17)(11.55,-1.86)
\psline{->}(11.55,-1.86)(10.85,-1.17)
\psline(4.02,1.42)(6.98,-0.54)
\end{pspicture*}
}
\caption{Construction of the new mesh $\M_{\varepsilon}$ and of the function $u_{\varepsilon}$.}
\label{fig1}
\end{figure}
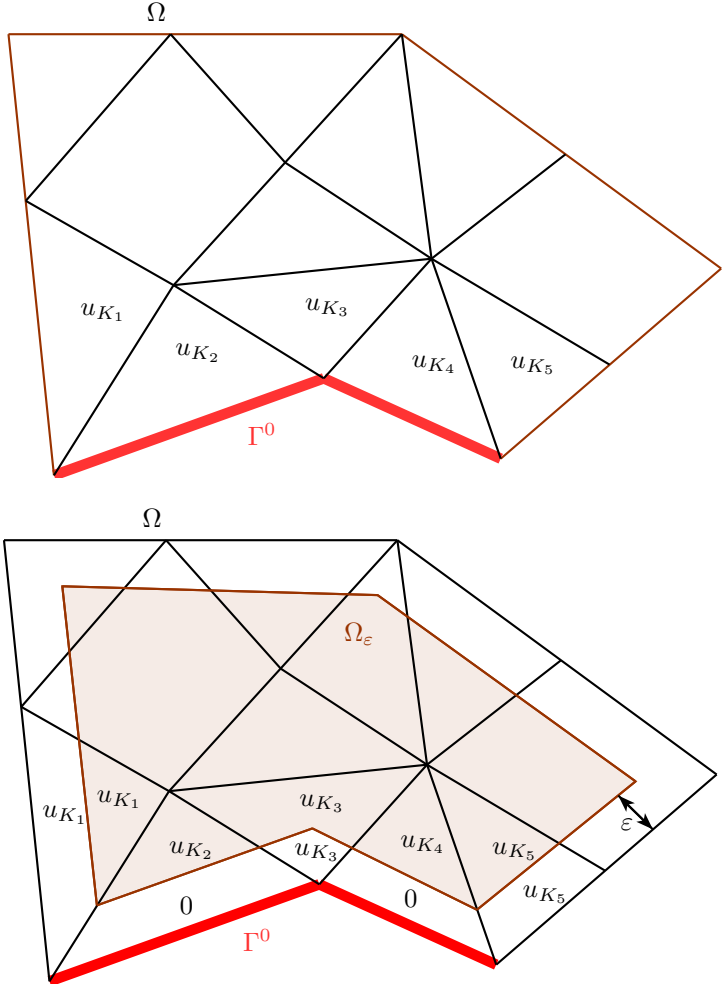
\end{proof}
\begin{rem} Observe that we could only modify the mesh locally around the
boundary $\Gamma_0$ but the procedure is more complicated to the present one.
Then, for sake of clarity, we prefer to add more control volumes
$K^\varepsilon$,   which are useless when we pass to the limit
$\varepsilon\rightarrow 0$.
\end{rem}

Now using this lemma we can prove the discrete Gagliardo-Nirenberg-Sobolev and Poincar\'e-Sobolev inequalities in the case with some homogeneous Dirichlet boundary conditions.



\subsection{Discrete Poincar\'e-Sobolev inequality}\label{sec_Sobolevdir}

In the case with some homogeneous Dirichlet boundary conditions, the discrete Poincar\'e-Sobolev inequalities are rewritten as follows.
\begin{thrm}[Discrete Poincaré-Sobolev inequality]
Let $ \Omega$ be an open connected bounded polyhedral domain of $ \mathbb{R}^{N}$ and let $\Gamma^{0} \subset \Gamma$, with $\text{m}(\Gamma^{0})>0$.
Let $ \M$ be a mesh satisfying (\ref{regmesh}). 
\begin{itemize}
\item If $1\leq p<N$, let $1\leq q\leq p^{*}=\ds\frac{pN}{N-p}$,
\item  if $p\geq N$, let $1\leq q <+\infty$. 
\end{itemize}
Then there exists a constant $C>0$ which only depends on $p$, $q$, $N$, $\Gamma^{0}$ and $ \Omega$ such that:
\begin{equation}
\Vert \, u \, \Vert_{0,q} \, \leq \, \frac{C}{\xi^{(p-1)/p}}\,\, |\,u\,|_{1,p,\Gamma^0,\M} \quad  \forall u\in X({\M})
\label{spdir}
\end{equation}
\label{thm2}
\end{thrm}

\begin{proof}
For $u \in X(\M)$, we apply Lemma \ref{lembase} to $v\in X(\M)$ defined by $v_{K}=|u_{K}|^{s}$ for all $K \in \M$. It yields
\begin{equation}
\Vert\,u\,\Vert_{0,sN/(N-1)}^{s}\,\leq\,\frac{C}{\xi^{(p-1)/p}}\,|\,u\,|_{1,p,\Gamma^{0},\M}\,\Vert\,u\,\Vert_{0,(s-1)p/(p-1)}^{(s-1)}
\label{ineg1psD}
\end{equation}
with $p>1$ and $s \geq 1$.\\
Then the proof is similar to the proof of Theorem \ref{thm4}, starting with inequality \eqref{BV2} instead of inequality \eqref{BV1} and using \eqref{ineg1psD} instead of \eqref{ineg1ps}.
\end{proof}


\subsection{Discrete Gagliardo-Nirenberg-Sobolev and Nash inequalities}\label{sec_GNSdir}

\begin{thrm}[Discrete Gagliardo-Nirenberg-Sobolev inequality]
Let $ \Omega$ be an open connected bounded polyhedral domain of $ \mathbb{R}^{N}$ and $\Gamma^{0}$ be a part of the boundary such that $\text{m}(\Gamma^{0}) >0$. Let $ \M$ be a mesh satisfying (\ref{regmesh}). 
\begin{itemize}
\item If $1\leq p<N$, let $1\leq s \leq m \leq p^{\ast}=\ds\frac{pN}{(N-p)}$,
\item  if $p\geq N$, let $1\leq s \leq m <+\infty$. 
\end{itemize}
Then, there exists a constant $C>0$ only depending on $p$, $s$, $N$, $\theta$ and $\Omega$ such that:
\begin{equation*}
\Vert \, u \, \Vert_{0,m}\, \leq \, \frac{C}{\xi^{(p-1)\theta/p}}\,\, |\, u \,|_{1,p,\Gamma^0,\M}^{\theta} \, \Vert \, u \, \Vert_{0,s}^{1-\theta}, \quad \forall u\in X({\M}),
\end{equation*}
where $\theta$ is given by \eqref{theta:opt}.
\label{thm1}
\end{thrm}

\begin{proof}
The proof is similar to the proof of Theorem \ref{thm3}, starting from \eqref{spdir} instead of \eqref{poincare3}.
\end{proof}

Now using Theorems \ref{thm2} and \ref{thm1}, we easily get a discrete version of Nash inequality:

\begin{cor}[Discrete Nash inequality]
Let $ \Omega$ be an open connected bounded polyhedral domain of $ \mathbb{R}^{N}$ and $\Gamma^{0}$ be a part of the boundary such that $\text{m}(\Gamma^{0}) >0$. Let $ \M$ be a mesh satisfying (\ref{regmesh}). 
Then there exists a constant $C>0$ only depending on $ \Omega$, $\Gamma^{0}$ and $N$ such that
\begin{equation*}
\Vert\, u \,\Vert_{0,2}^{1+\frac{2}{N}} \, \leq \, \frac{C}{\sqrt{\xi}} \,\, |\,u\,|_{1,2,\Gamma^0,\M} \, \Vert\, u \,\Vert^{\frac{2}{N}}_{0,1}\quad \forall u\in X({\M}).
\end{equation*}
\end{cor}

\section{Application to finite volume approximations coming\\ from DDFV schemes}\label{sec.ddfv}

The discrete duality finite volume methods have been developed for ten years for the approximation of anisotropic elliptic problems on almost general meshes in 2D and 3D. They are based on some discrete operators (divergence and gradient), satisfying a discrete Green formula (the ``discrete duality''). The DDFV approximations were first proposed for the discretization of anisotropic and/or nonlinear diffusion problems on rather general meshes. We refer to the pioneer work of F. Hermeline \cite{Hermeline1998,Hermeline2000,Hermeline2004,Hermeline2007,Hermeline2009} who proposed a new approach dealing with primal and dual meshes   and  Y. Coudi\`ere, J.-P. Vila anf Ph. Villedieu \cite{Coudiere1999} who proposed a method of reconstruction for the discrete gradients. Next, K. Domelevo and P. Omn\`es \cite{Domelevo2005}, S. Delcourte, K. Domelevo and P. Omn\`es \cite{Delcourte2005} presented the discrete duality finite volume approach (DDFV) for the Laplace operator. Then, B. Andreianov, F. Boyer and F. Hubert \cite{Andreianov2007} gave a general background of DDFV methods for anisotropic and nonlinear elliptic problems. Most of these works treat 2D linear anisotropic, heterogeneous diffusion problems, while the case of discontinuous diffusion operators have been treated later by F. Boyer and F. Hubert in \cite{Boyer2008}. F. Hermeline \cite{Hermeline2007,Hermeline2009}, Y. Coudi\`ere and F. Hubert \cite{CoudiereHubert}, B. Andreianov, M. Bendahmane, F. Hubert and S. Krell \cite{ABHK} treat the analogous 3D problems. S. Krell  in \cite{Krell2011} treats the Stokes problem in 2D and in 3D whereas Y. Coudi\`ere and G. Manzini in \cite{Coudi`ere2010} treat linear elliptic convection-diffusion equations.

Our aim is now to extend the results from Sections \ref{sec.gc} and \ref{sec.dirichlet} to finite volume approximations coming from DDFV schemes in 2D.  
%
\subsection{Meshes and functional spaces}\label{sec.ddfv.meshes}

The construction of DDFV schemes needs the definition of three meshes: a primal mesh, a dual mesh and a diamond mesh. Then, the approximate solutions are defined both on the primal and the dual meshes, while the approximate gradients are defined on the diamond mesh. Therefore, we need to adapt the definition of the spaces of approximate solutions and the definition of the discrete norms.

\paragraph{Meshes.}

Let $\Omega$ be an open bounded  polygonal domain of $\R^2$. The mesh construction starts with the partition
of  $\Omega$ with disjoint open polygonal control volumes.
This partition, denoted by  $\M$, is called the interior primal mesh.
 We denote by $\dr\M$ the set of boundary edges, which are considered as degenerate control volumes. Then, the primal mesh is defined by  $\overline{\M}=\M\cup\dr\M$.

By connecting the centers of the primal mesh, we get a dual mesh $\overline{\Mie}$, in which we can distinguish the interior dual mesh $\M^*$ and the exterior dual mesh $\dr\M^*$, $\overline{\Mie}= \M^*\cup\dr\M^*$.
In the sequel, we will assume that  each primal cell $\k\in\M$ is star-shaped with respect to its center and each dual cell $\ke\in\overline{\Mie}$ 
is star-shaped with respect to its center.

We denote by ${\mathcal E}$ the set of edges of the primal mesh and ${\mathcal E}^*$ the set of edges of the dual mesh. To each edge of the primal mesh $\sigma\in \E$, we associate a diamond defined by connecting the vertices of the edge and the centers of the primal cells sharing the edge. The diagonals of this diamond are $\sigma\in\E$ and $\sigma^*\in\E^*$; we may note it $\Dsig$.  If $\sigma\in \E\cap\dr\Omega$,
 we note that the diamond degenerates into a triangle. The set of the diamond cells defines  a partition of $\Omega$, which is called the diamond mesh and is denoted by $\DD$. Let us note that $\DD$ can be splitted into $\DD=\DD_{int}\cup \DD_{ext}$ where $\DD_{int}$ is the set of interior (non degenerate) diamond cells and $\DD_{ext}$ is the set of degenerate diamond cells. 

%
 Finally, the DDFV mesh is made of the triple $\T=(\overline{\M},\overline{\Mie},\DD)$. 

\begin{figure}[htb]
\begin{center}
\begin{tikzpicture}[scale=1.2]
\node[rectangle,scale=0.8,fill=black!50] (xle) at (0,0) {};
\node[circle,draw,scale=0.5,fill=black!5] (xl) at (2,1.3) {};
\node[rectangle,scale=0.8,fill=black!50]  (xke) at (0,4) {};
\node[circle,draw,scale=0.5,fill=black!5] (xk) at (-2,2.3) {};
\draw[line width=1pt] (xle)--(xke);
\draw[dashed, line width=1pt] (xk)--(xl);
\draw[dash pattern=on 2pt off 3pt on 6pt off 3pt,line width=2pt] (xk)--(xke)--(xl)--(xle)--(xk);

\node[yshift=-8] at (xle){$\xle$};
\node[yshift=8] at (xke){$\xke$};
\node[xshift=10] at (xl){$\xl$};
\node[xshift=-10] at (xk){$\xk$};

\draw[->,line width=1pt] (-2+3.*0.4,2.3-3.*0.1)--(-2+4.8*0.4,2.3-4.8*0.1);
\draw[->,line width=1pt] (-2+3.*0.4,2.3-3.*0.1)--(-2+3.*0.4-1.8*0.1,2.3-3.*0.1-1.8*0.4);
\draw[->,line width=1pt] (0,2.7)--(0,2.);
\draw[->,line width=1pt] (0,2.7)--(.7,2.7);
\node[right] at (0,2.3){$\tkele$};
\node[right] at (0.1,2.95){$\nksig$};
\node at (-0.3,1.6){$\tkl$};
\node at (-0.5,1.2){$\nkesige$};
\draw (0,1.5) arc (-90:-10:0.3);
\node[right] at (0.1,1.4){$\alpha_\D $};

\node[rectangle,scale=0.8,fill=black!50] at (3.,2) {};
\node[circle,draw,scale=0.5,fill=black!5] at (3.,1.5) {};
\node[right,xshift=8] at (3.2,2){Vertices of the primal mesh};
\node[right,xshift=8] at (3.2,1.5){Centers of the primal mesh};
\draw[line width=1pt]  (2.6,1)--(3.2,1);
\node[right,xshift=8] at (3.2,1){$\sigma=\k\vert\l$, edge of the primal mesh};
\draw[dashed, line width=1pt] (2.6,0.5)--(3.2,0.5);
\node[right,xshift=8] at (3.2,.5){$\sigma^*=\ke\vert\le$, edge of the dual mesh};
\draw[dash pattern=on 2pt off 3pt on 6pt off 3pt,line width=2pt] (2.6,0)--(3.2,0);
\node[right,xshift=8] at (3.2,.){Diamond $\Dsig$};

\node[rectangle,scale=0.8,fill=black!50] (xle2) at (9,0) {};
\node[circle,draw,scale=0.5,fill=black!5] (xl2) at (9,2) {};
\node[rectangle,scale=0.8,fill=black!50]  (xke2) at (9,4) {};
\node[circle,draw,scale=0.5,fill=black!5] (xk2) at (7.5,2.3) {};
\draw[line width=1pt] (xle2)--(xke2);
\draw[dashed, line width=1pt] (xk2)--(xl2);
\draw[dash pattern=on 2pt off 3pt on 6pt off 3pt,line width=2pt] (xle2)--(xk2)--(xke2);

\node[yshift=-8] at (xle2){$\xle$};
\node[yshift=8] at (xke2){$\xke$};
\node[xshift=10] at (xl2){$\xl$};
\node[xshift=-10] at (xk2){$\xk$};

\node[right] at (9.6,3){$\dkel$};
\node[right] at (9.6,1){$\dlel$};
\draw[<->,line width=0.8pt] (9.6,2.)--(9.6,4);
\draw[<->,line width=0.8pt] (9.6,0.)--(9.6,2);
\end{tikzpicture}
\end{center}
\caption{Definition of the diamonds $\Dsig$}\label{fig_diamonds}
\end{figure}

Let us now introduce some notations associated to the mesh $\T$. For each primal cell or dual cell $V$ in ${\overline \M}$ or $\overline{\Mie}$, we define 
 $\E_V$, the set of edges of $V$,  $\DD_V=\{\Dsig\in\DD,\ \sig\in\E_V\}$,
$d_V$, the diameter of $V$.
For a diamond $\D$, whose vertices are $(\xk,\xke,\xl,\xle)$, we define  
$d_\D$ its diameter and 
$\alpha_\D$ the angle between $(\xk,\xl)$ and $(\xke,\xle)$.
As shown on Figure \ref{fig_diamonds}, we will also use two direct basis $(\tkele,\nksig)$ and $(\nkesige,\tkl)$, where 
$\nksig$ is the unit normal to $\sigma$, outward $\k$,
$\nkesige$ is the unit normal to $\sigma^*$, outward $\ke$,
$\tkele$ is the unit tangent vector to $\sigma$, oriented from $\ke$ to $\le$,
$\tkl$ is the unit tangent vector to $\sigma^*$, oriented from $\k$ to $\l$.
%

In all the sequel, we will assume that the diamonds cannot be flat. It means :
\begin{equation}\label{def.alphat}
\exists \alpha_\petitt \in ]0,\frac{\pi}{2}] \mbox { such that } |\sin(\alpha_\D)|\geq \sint\quad  \forall \D\in\DD.
\end{equation}
As for all $\D=\Dsig\in\DD$, we have $2{\rm m}(\D)={\rm m}(\sigma)\msige\sin(\alpha_\D)$ and hypothesis \eqref{def.alphat} implies
\begin{equation*}
{\rm m}(\sigma)\msige\leq \ds\frac{2\md}{\sin(\alpha_\petitt)}.
\end{equation*}
We also assume some regularity of the mesh, as in \cite{Andreianov2007}, which implies
\begin{equation}\label{def.zeta}
\begin{aligned}
\exists \zeta>0,& \ds\sum_{\Dsig\in\DD_{\k}}\!\!{\rm m}(\sigma)\msige\leq \frac{\mk}{\zeta}\quad  \forall \k\in \M,\\
&\ds\sum_{\Dsig\in\DD_{\ke}}\!\!{\rm m}(\sigma)\msige\leq \frac{\mke}{\zeta}\quad  \forall \ke\in \overline{\Mie}.
\end{aligned}
\end{equation}

\paragraph{Definition of the approximate solution.}

A discrete duality finite volume scheme leads to the computation of discrete unknowns on the primal and the dual meshes :
$(\uk)_{\k\in \overline{\M}}$ and $(\uke)_{\ke\in \overline{\Mie}}$. From these discrete unknowns, we can reconstruct two different approximate solutions : 
$$
u_{\M}=\ds\sum_{\k\in\M} \uk {\mathbf 1}_{\k} \mbox{ and } u_{\overline{\Mie}}=\ds\sum_{\ke\in\overline{\Mie}} \uke {\mathbf 1}_{\ke}.
$$
But, in order to use simultaneously the discrete unknowns computed on the primal and the dual meshes, we prefer to define the approximate solution as 
$
u=\frac{1}{2}(u_{\M}+u_{\overline{\Mie}}).
$
Therefore, the space of approximate solutions $Z(\T)$ is defined by:
\begin{multline*}
Z(\T)=\left\{\vphantom{\left(\ds\sum_{\k\in\M} \uk {\mathbf 1}_{\k}\right)}u\in L^1(\Omega)\ / \ \exists u_{\T}=\left( (\uk)_{\k\in \overline{\M}},(\uke)_{\ke\in \overline{\Mie}}\right) \right.\\[-4.mm]
\left.\mbox{ such that } u= \frac{1}{2}\left(\ds\sum_{\k\in\M} \uk {\mathbf 1}_{\k}+\ds\sum_{\ke\in\overline{\Mie}} \uke {\mathbf 1}_{\ke}\right) \right\}.
\end{multline*}
For a given function $u\in Z(\T)$, we define the discrete $L^p$-norm $
\Vert u\Vert _{0,p,\T}$ by 
$$
\Vert u\Vert _{0,p,\T}^p=\left(\frac{1}{2}\Vert u_{\M}\Vert_{0,p}^p+\frac{1}{2}\Vert u_{\Mie}\Vert_{0,p}^p\right).
$$

\paragraph{Discrete gradient.}

A key point in the construction of the DDFV schemes is the definition of the discrete operators (divergence and gradient). We just focus here on the definition of the discrete gradient, which will be useful for the definition of the discrete $W^{1,p}$-seminorms.

Let $u\in Z(\T)$. The discrete gradient of $u$, $\nabla^d u$ is defined as a piecewise constant function on each diamond cell :
$$
\nabla^d u=\ds\sum_{\D\in\DD} \gradD u\,{\mathbf 1}_{\D},
$$
where, for $\D\in\DD$, 
$$
\gradD u=\frac{1}{\sind}\left(\frac{\ul-\uk}{\msige}\nksig+\frac{\ule-\uke}{{\rm m}(\sigma)}\nkesige\right).
$$
This discrete gradient has been introduced in \cite{Coudiere1999}. It verifies:
$$
 \gradD u\cdot\tkele=\frac{\ule-\uke}{{\rm m}(\sigma)} \mbox{ and } \gradD u\cdot\tkl=\frac{\ul-\uk}{\msige}.
$$
Using this discrete gradient, we may now define the discrete $W^{1,p}$-seminorm and norm of a given function $u\in Z(\T)$:
\begin{eqnarray*}
|u|_{1,p,\T}&=&\left(\ds\sum_{\D\in\DD} \md|\gradD u|^p\right)^{1/p},  \\
\Vert u\Vert_{1,p,\T} &=& \Vert u\Vert _{0,p,\T}+|u|_{1,p,\T}.
\end{eqnarray*}

\subsection{Discrete functional inequalities in the general case}\label{sec.ddfv.gen}

Our aim is now to extend the results of Section \ref{sec.gc} to the case of finite volume approximations coming from some DDFV schemes: 
$u\in Z(\T)$. We will use that such functions are defined as $u=\frac{1}{2}(u_{\M}+u_{\overline{\Mie}})$ with $u_{\M}\in X(\M)$ 
and $u_{\overline{\Mie}}\in X({\overline{\Mie}})$. 
 
\begin{thrm}[General discrete Poincaré-Sobolev inequality in the DDFV framework]
Let $ \Omega$ be an open bounded polygonal domain of $ \mathbb{R}^{2}$. Let $\T=(\overline{\M},\overline{\Mie},\DD)$ be a DDFV mesh satisfying \eqref{def.alphat} and \eqref{def.zeta}.
\begin{itemize}
\item If $1 \leq p<2$, let  $1 \leq q \leq p^{*}=\ds\frac{2p}{2-p}$,
\item  if $p \geq 2$, let $1 \leq q < + \infty$.
\end{itemize}
Then there exists a constant $C>0$ only depending on $p$, $q$ and $\Omega$ such that:
\begin{equation*}
\Vert\, u \,\Vert_{0,q,\T} \, \leq \, \ds\frac{C}{(\sin(\alpha_\petitt))^{1/p}\zeta^{(p-1)/(p)}}\,\, \Vert\, u \,\Vert_{1,p,\T},\quad \forall u\in Z(\T),
\end{equation*}
\label{thm4DDFV}
\end{thrm}
 
\begin{proof}
Since the proof is similar to the proof of Theorem \ref{thm4}, we only detail here the case $1<p<2$. Let $s\geq 1$. For $u\in Z(\T)$, as $u_{\M}\in X(\M)$ and $u_{\overline{\Mie}}\in X({\overline{\Mie}})$, we may write:
\begin{eqnarray}
 \left\Vert u_\M \right\Vert_{0,2s,\M}^s&\leq& c(\Omega) \left(\bigl| |u_\M|^s\bigl|_{1,1,\M}+ \Vert u_\M \Vert_{0,s,\M}^s\right)\\
  \left\Vert u_{\overline{\Mie}}\right\Vert_{0,2s,{\overline{\Mie}}}^s&\leq& c(\Omega) \left(\bigl||u_{\overline{\Mie}}|^s\bigl|_{1,1,{\overline{\Mie}}}+ \Vert u_{\overline{\Mie}} \Vert_{0,s,{\overline{\Mie}}}^s\right)\label{depmstar}
 \end{eqnarray}
But, following the same computations as in the proof of Theorem \ref{thm4}, we get 
\begin{multline*}
 \bigl| |u_\M|^s\bigl|_{1,1,\M}= \ds\sum_{\Dsig\in\DD_{int}} {\rm m}(\sigma)\Bigl\vert|\uk|^s-|\ul|^s\Bigl\vert\\
 \leq\left(  \ds\sum_{\Dsig\in\DD_{int}}{\rm m}(\sigma)\msige \left\vert \frac{\uk-\ul}{\msige}\right\vert^p   \right)^{\frac{1}{p}} \left( \ds\sum_{\k\in\M}\sum_{\Dsig\in\DD_{\k}} {\rm m}(\sigma)\msige |\uk|^{\frac{(s-1)p}{p-1}}\right)^{\frac{p-1}{p}}
 \end{multline*}
 Using the regularity hypotheses on the mesh, we get 
 $$
 \bigl| |u_\M|^s\bigl|_{1,1,\M}\leq \ds\frac{C}{(\sin(\alpha_\petitt))^{1/p}\zeta^{(p-1)/p}} \left(  \ds\sum_{\Dsig\in\DD_{int}}\md \left\vert \frac{\uk-\ul}{\msige}\right\vert^p   \right)^{\frac{1}{p}} \Vert u_\M\Vert_{0,(s-1)p/(p-1),\M}^{s-1}.
 $$
  But, by definition, $\ds\frac{\uk-\ul}{\msige}=\gradD u\cdot\tkl$ and therefore $\ds\left\vert \frac{\uk-\ul}{\msige}\right\vert\leq \vert\gradD u\vert$. It yields :
 $$
 \bigl| |u_\M|^s\bigl|_{1,1,\M}\leq \ds\frac{C}{(\sin(\alpha_\petitt))^{1/p}\zeta^{(p-1)/p}}|u|_{1,p,\T}\lV u_\M\rV_{0,(s-1)p/(p-1),\M}^{s-1}.
 $$
 Then it yields that for any $p>1$ and $s\geq 1$
 \begin{equation*}
\Vert\,u_{\M}\,\Vert_{0,2s,\M}^{s}\,\leq\,C\left(\frac{1}{\sin(\alpha_\petitt)^{1/p}\zeta^{(p-1)/p}}\,|\,u\,|_{1,p,\T}\,\Vert\,u_{\M}\,\Vert_{0,(s-1)p/(p-1),\M}^{s-1}+\Vert\,u_{\M}\,\Vert_{0,s,\M}^{s}\right).
\end{equation*}
Thus using interpolation inequality \eqref{ineg2ps} and choosing $s=p/(2-p)\geq 1$, we obtain as in the proof of Theorem \ref{thm4}
\begin{equation*}
\Vert\,u_{\M}\,\Vert_{0,2p/(2-p),\M}\,\leq\,C\left(\frac{1}{\sin(\alpha_\petitt)^{1/p}\zeta^{(p-1)/p}}\,|\,u\,|_{1,p,\T}+\Vert\,u\,\Vert_{0,p,\M}\right)
\end{equation*}
and finally since $\Vert\,u_{\M}\,\Vert_{0,p,\M}\leq 2\Vert\,u\,\Vert_{1,p,\T}$ by definition, we get
\begin{equation*}
\Vert\,u_{\M}\,\Vert_{0,2p/(2-p),\M}\,\leq \, \frac{C}{\sin(\alpha_\petitt)^{1/p}\zeta^{(p-1)/p}}\,\Vert\,u\,\Vert_{1,p,\T}.
\end{equation*}
With similar computations on the dual mesh, from \eqref{depmstar} we get
\begin{equation*}
\Vert\,u_{\overline{\Mie}}\,\Vert_{0,2p/(2-p),\overline{\Mie}}\,\leq \, \frac{C}{\sin(\alpha_\petitt)^{1/p}\zeta^{(p-1)/p}}\,\Vert\,u\,\Vert_{1,p,\T}.
\end{equation*}
Then we can conclude using inclusion of $L^{2p/(2-p)}(\Omega)$ in $L^{q}(\Omega)$ for all $q \leq p^{*}$.
\end{proof}

 As in the classical finite volume framework, we can now prove discrete Gagliardo-Nirenberg-Sobolev inequalities. The proof is similar to the proof of Theorem \ref{thm3}; it will not be detailed here.

\begin{thrm}[General discrete Gagliardo-Nirenberg-Sobolev inequality in the DDFV framework]
Let $ \Omega$ be an open bounded polygonal domain of $ \mathbb{R}^{2}$. Let $\T=(\overline{\M},\overline{\Mie},\DD)$ be a DDFV mesh satisfying \eqref{def.alphat} and \eqref{def.zeta}.
\begin{itemize}
\item If $1 \leq p <2$, let $1\leq s\leq m\leq p^{\ast}=\frac{2p}{2-p}$,
\item if $1 \leq p <2$, let $1\leq s\leq m<+\infty$.
\end{itemize}
Then, there exists a constant $C>0$ only depending on $p$, $s$, $m$ and $\Omega$ such that:
\begin{equation*}
\Vert\, u \,\Vert_{0,m,\T} \, \leq \, \ds\frac{C}{(\sin(\alpha_\petitt))^{\theta/p}\zeta^{\theta(p-1)/(p)}}\,\, \Vert\, u \,\Vert^{\theta}_{1,p,\T}\, \Vert \,u \,\Vert_{0,s,\T}^{1-\theta},\ \forall u\in Z(\mathcal{T}),
\end{equation*}
where $\theta$ is given by  \eqref{theta:opt}.
\end{thrm}

 Let us now focus on the Poincar\'e-Wirtinger inequality in the DDFV case. This result has been proved recently in \cite{Le}. We will give here a very short proof using the embedding of $BV(\Omega)$ into $L^2(\Omega)$ recalled in Theorem \ref{thmBV2}.
 \begin{thrm}[Discrete Poincar\'e-Wirtinger inequality in the DDFV framework]
 Let $ \Omega$ be an open bounded connected polygonal domain of $ \mathbb{R}^{2}$. Let $\T=(\overline{\M},\overline{\Mie},\DD)$ be a DDFV mesh satisfying \eqref{def.alphat}.
 There exists a constant $C>0$ depending only on $\Omega$, such that for all $u\in Z(\T)$ satisfying 
 \begin{equation}\label{hypmoynulle}
 \ds\sum_{\k\in\M} \mk \uk= \ds\sum_{\k\in\overline{\Mie}} \mke \uke=0,
 \end{equation}
 we have 
 \begin{equation*}
 \lV u\rV_{0,2,\T}\leq \ds\frac{C}{\sin(\alpha_\petitt)} \lv u\rv_{1,2,\T}.
 \end{equation*}
 \end{thrm}
 
 \begin{proof}
 Let $u\in Z(\T)$. Applying \eqref{BV3} to $u_\M\in X(\M)$ and $u_{\overline{\Mie}}\in X(\overline{\Mie})$, we get, under the hypothesis \eqref{hypmoynulle},
$$ 
\lV u_\M\rV_{0,2,\M} \,+\, \lV u_{\overline{\Mie}}\rV_{0,2,{\overline{\Mie}}} \,\leq\, c(\Omega)\, \left(\, \lv u_\M\rv_{1,1,\M} \,+\,  \lv u_{\overline{\Mie}}\rv_{1,1,{\overline{\Mie}}} \;\right).
 $$
 But,
 \begin{eqnarray*}
 \lv u_\M\rv_{1,1,\M}&\leq& \ds\sum_{\Dsig\in\DD_{int}}{\rm m}(\sigma)\msige\frac{|\uk-\ul|}{\msige}\\
 &\leq&\ds \frac{2}{\sin(\alpha_\petitt)}\sum_{\Dsig\in\DD_{int}}\md\frac{|\uk-\ul|}{\msige}\\
 &\leq &\ds \frac{2}{\sin(\alpha_\petitt)}\textrm{m}(\Omega)^{1/2}\lv u\rv_{1,2,\T},
 \end{eqnarray*}
 thanks to Cauchy-Schwarz inequality. By the same way, 
 we get  the same bound for $ \lv u_{\overline{\Mie}}\rv_{1,1,{\overline{\Mie}}}$ and it finally yields $ \lV u\rV_{0,2,\T}\leq \ds \frac{2}{\sin(\alpha_\petitt)}\textrm{m}(\Omega)^{1/2}c(\Omega)\lv u\rv_{1,2,\T}$.
 \end{proof}
 
 \subsection{Discrete functional inequalities in the case with Dirichlet boundary conditions}\label{sec.ddfv.dir}

In this Section, we want to extend the discrete Poincaré-Sobolev inequalities of Section \ref{sec_GNSdir} to finite volume approximations obtained from a DDFV scheme. We first recall how Dirichlet boundary conditions are taken into account in DDFV methods. Let $\Gamma^0$ be a part of the boundary such that $\text{m}(\Gamma^{0}) >0$. At the discrete level, homogeneous Dirichlet boundary conditions on $\Gamma^0$ will be written:
\begin{equation}\label{DDFVbc}
\uk=0,\  \forall \k\in\dM,\ \k\subset\Gamma^0 \mbox{ and }  \uke=0,  \  \forall \ke\in\dMie,\ \overline{\ke}\cap\Gamma^0\neq\emptyset.
\end{equation}
Therefore,we consider the corresponding set of finite volume approximations, $Z_{\Gamma^0}(\T)$ defined by:
$$
Z_{\Gamma^0}(\T)=\left\{ u\in Z(\T) \mbox{ satisfying } \eqref{DDFVbc}\right\}.
$$
Let us note that the definition of the discrete $W^{1,p}$-seminorm is the same on $Z_{\Gamma^0}(\T)$ as on $Z(\T)$. Indeed, the fact that the approximate solution vanishes at the boundary is taken into account in the definition of the discrete gradient $\gradD u$ for $\D\in\DD_{ext}$, and therefore in $\lv u\rv_{1,p,\T}$.

Finally, combining the techniques of proof of Theorem \ref{thm2} (using Lemma \ref{lembase}) and Theorem \ref{thm4DDFV}, we establish the following Theorem.

\begin{thrm}[Discrete Poincaré-Sobolev inequality in the DDFV framework]
Let $ \Omega$ be an open connected bounded polygonal domain of $ \mathbb{R}^{2}$ and $\Gamma^0$ be a part of the boundary such that ${\textrm m}(\Gamma^0)>0$. Let $\T=(\overline{\M},\overline{\Mie},\DD)$ be a DDFV mesh satisfying \eqref{def.alphat} and \eqref{def.zeta}.
\begin{itemize}
\item  If $1 \leq p<2$, let $1 \leq q \leq p^{*}=\ds\frac{2p}{2-p}$,
\item if $p \geq 2$, let $1 \leq q < + \infty$.
\end{itemize}
Then there exists a constant $C>0$ only depending on $p$, $q$, $\Gamma^{0}$ and $\Omega$ such that:
\begin{equation*}
\Vert\, u \,\Vert_{0,q,\T} \, \leq \, \ds\frac{C}{(\sin(\alpha_\petitt))^{1/p}\zeta^{(p-1)/(p)}}\,\, \vert\, u \,\vert_{1,p,\T},\quad \forall u\in Z(\T),
\end{equation*}
\end{thrm}

\textbf{Acknowledgement:} This work was partially supported by the European Research Council ERC Starting Grant 2009, project 239983-NuSiKiMo. The authors want to thank one anonymous referee for her/his careful reading and suggestions which have improved a lot this paper.

\bibliographystyle{plain}
\bibliography{bibliographie}

\end{document}